\documentclass{article}

\usepackage{amssymb,amsmath,mathrsfs}
\usepackage[colorlinks=false]{hyperref}
\usepackage{comment}
\usepackage{diagbox}

\usepackage{geometry}
\usepackage{diagbox}
\usepackage{manyfoot}
\usepackage{graphicx,color}
\usepackage[outercaption]{sidecap}
\usepackage{xfrac}
\usepackage{subfig}
\usepackage{tikz}
\usepackage{enumitem}
\usepackage{pdfsync}
\usepackage{fancyhdr}
\usepackage[nottoc,notlof,notlot]{tocbibind}
\usepackage[titles,subfigure]{tocloft}

\usepackage{amssymb}
\usepackage{amsmath}
\usepackage{latexsym}
\usepackage{amsthm}
\usepackage{eucal}
\usepackage{amsthm}
\theoremstyle{plain}
\newtheorem{theorem}{Theorem}[section]
\newtheorem{corollary}[theorem]{Corollary}
\newtheorem{lemma}[theorem]{Lemma}
\newtheorem{proposition}[theorem]{Proposition}
\theoremstyle{definition}

\newtheorem{remark}[theorem]{Remark}
\newtheorem*{ack}{Acknowledgments}

\def\e{\varepsilon}
\def\rr{{\mathbb R}}

\def\NN{{\mathbb N}}
\def\ZZ{{\mathbb Z}}

\def\dist{{\rm dist}}

\def\supp{{\rm supp}}
\def\loc{{\rm loc}}

\def\Xint#1{\mathchoice 
  {\XXint\displaystyle\textstyle{#1}}%
  {\XXint\textstyle\scriptstyle{#1}}%
  {\XXint\scriptstyle\scriptscriptstyle{#1}}%
  {\XXint\scriptscriptstyle\scriptscriptstyle{#1}}%
  \!\int} 
\def\XXint#1#2#3{{\setbox0=\hbox{$#1{#2#3}{\int}$} 
  \vcenter{\hbox{$#2#3$}}\kern-.5\wd0}} 
\def\mint{\Xint -}

\newcommand{\minitab}[2][l]{\begin{tabular}#1 #2\end{tabular}}

\numberwithin{equation}{section}

\title{Variational analysis of nonlocal Dirichlet problems in periodically perforated domains }
\author{
{\sc Roberto Alicandro}
\\ \small Dipartimento di Matematica ``R. Caccioppoli'',
 Universit\`a di Napoli Federico II,\\
 \small
via Cintia, 80126 Napoli, Italy\\
\\
{\sc Maria Stella Gelli}
\\ \small Dipartimento di Matematica, Universit\`a di Pisa,\\
\small Largo Bruno Pontecorvo, 56127 Pisa, Italy\\
\\
{\sc Chiara Leone}
\\ \small Dipartimento di Matematica ``R. Caccioppoli'',
 Universit\`a di Napoli Federico II,\\
 \small
via Cintia, 80126 Napoli, Italy\\
}
\date{
}                                      

\begin{document}

\maketitle

\noindent
{\bf Abstract.}
In this paper we consider a family of non local functionals of convolution-type depending on a small parameter $\e>0$ and $\Gamma$-converging to local functionals defined on Sobolev spaces as $\e\to 0$. We study the asymptotic behaviour of the functionals when the order parameter is subject to Dirichlet conditions on a periodically perforated domains, given by a periodic array  of small balls of radius $r_\delta$  centered on a $\delta$--periodic lattice, being $\delta > 0$ an additional small parameter and $r_\delta=o(\delta)$. We highlight differences and analogies with the local case, according to the interplay between the three scales $\e$, $\delta$ and $r_\delta$. A fundamental tool in our analysis turns out to be a non local variant of the classical   Gagliardo-Nirenberg-Sobolev inequality in Sobolev spaces which may be of independent interest and useful for other applications.

\bigskip



\section{Introduction}
In the last decades there has been an increasing interest towards the analysis of variational models involving non local functionals of the form
\begin{equation}\label{doubleint}
\int_\Omega\int_\Omega f(x, y, u(y)-u(x))\, dx\, dy
\end{equation}
where $\Omega$ is an open set of $\rr^d$, in view of their relevance for applications in different directions, such as image processing \cite{brengu, GO}, population dynamics \cite{FKK}, continuum mechanics through the theory of perydinamics \cite{BMCP, MD, Sil} and phase transition problems \cite{albel,SV}.

\noindent The relation between non local functionals as in \eqref{doubleint} when the energies concentrate on the diagonal $x=y$  and local functionals of the form
\begin{equation}\label{locfunct}
\int_\Omega f(x, \nabla u(x))\, dx
\end{equation}
has been first investigated by  Bourgain, Brezis and Mironescu in their seminal paper \cite{boubremir}, where they study the asymptotic behaviour of Gagliardo seminorms $\displaystyle [u]_{W^{1-\e,p}(\Omega)}$ as $\e\to 0$, and in particular, in the case $p=2$, show  that
$$
\e[u]_{W^{1-\e,2}(\Omega)}=\e\int_\Omega\int_\Omega \frac{\ |u(y)-u(x)|^2\  }{\ |y-x|^{d+2(1-\e)}}\, dx\, dy 
$$ 
approximate as $\e\to 0$ the square of the $L^2$ norm of $\nabla u$, up to a multiplicative constant. The result has been subsequently extended in \cite{pon} in terms of $\Gamma$-convergence. A general asymptotic analysis  as $\e\to 0$ of families of functionals of the form
\begin{equation}\label{doubleinteps}
\int_\Omega\int_\Omega f_\e(x, y, u(y)-u(x))\, dx\, dy,
\end{equation}
under superlinear growth assumptions in the last variable and concentration of the energies on $x=y$, has been recently provided in \cite{AABPT}, by using De Giorgi localization methods for $\Gamma$-convergence,  leading to a general class of energies whose $\Gamma$-limits are of the form \eqref{locfunct},  with a number of applications, in particular to stochastic homogenization, to energies on point clouds and to gradient flows, which are just some of the potential directions of the theory.

Purpose of this paper is to investigate the asymptotic behaviour of energies as in \eqref{doubleinteps} when the order parameter $u$ is subject to pinning conditions, highlighting differences and analogies with the corresponding local case. Pinning sites are usually modelled as small zones where Dirichlet conditions are imposed. Here we consider the simplest case (but already presenting most of the main features) of periodically perforated domains where homogeneous Dirichlet conditions are imposed on a periodic array $P_\delta$ of small balls of radius $r_\delta$ centered on a $\delta$-periodic lattice, being $\delta>0$ an additional small parameter and $r_\delta=o(\delta)$. In the local case there is a wide literature devoted to the study of minimum problems involving energies as in \eqref{locfunct} subject to this type of constraints  and comprising a number of generalizations which  cover also  general non-periodic geometries. In particular, we refer to the celebrated paper by Cioranescu and Murat \cite{CM}, where the authors study the asymptotic behaviour, as $\delta\to 0$,  of solutions to the minima of the Dirichlet energy subject to the constraint above, and the paper \cite{AB}, where the problem is set in the framework of $\Gamma$-convergence and extended to general vector energies. The asymptotic description of the problems becomes non trivial when the growth of $f$ in the gradient variable in \eqref{locfunct}  is of order $p\leq d$ and leads to a critical size of the radii of the perforations, namely $r_\delta= O(\delta^{d/(d-p)})$, if $p<d$, and 
$\log(r_\delta)= O( \delta^{d/(d-1)})$, if $p=d$. Under this scaling the energetic contribution near each of the small balls can be decoupled from the others and from the diffused energy elsewhere and can be computed by means of a capacitary formula. For instance, in the model case $f(x,z)=|z|^p$ and given a forcing term $g\in L^{p'}(\Omega)$, one obtains that minimum problems
$$
\min\left\{\int_\Omega (|\nabla u|^p- gu)\, dx:\ u=0\ \hbox{on }\ P_\delta\right\}
$$
are approximated as $\delta\to 0$ by 
$$
\min\left\{\int_\Omega ( |\nabla u|^p\, dx +C_p|u|^p-gu)\, dx \right\},
$$
where $C_p$ is the $p$-capacity of the ball $B_1$ in $\rr^d$ (see Section \ref{preli}) 
and the middle term accounts for the energetic contribution near the perforations.

In this paper we focus on the case $p<d$ and consider energies defined on vector-valued functions $u\in L^p(\Omega;\rr^m)$ and of the form
$$
E_\e(u)=\frac{1}{\e^d}\int_\Omega\int_\Omega f\left(\frac{y-x}{\e}, \frac{u(y)-u(x)}{\e}\right)\, dx\, dy
$$
where, for any $\xi\in\rr^d$, $f(\xi,\cdot)$ is $p$-homogeneous, locally Lipschitz continuous and  satisfies suitable decay assumptions as $|\xi|\to +\infty$, ensuring that interactions between points $x$ and $y$ at a long range distance are negligible as $\e\to 0$ (see hypotheses ({\bf H}), ({\bf G}) and ({\bf L}) below). Under this assumptions, as a particular case of the asymptotic analysis provided in \cite{AABPT}, the $\Gamma$-limit of $E_\e$ is given by the local functional
$$
E_0(u)=\int_\Omega f_{hom}(\nabla u)\, dx,
$$
where $f_{hom}(S)$ is defined by a suitable homogenization formula (see \eqref{homform}). We then impose the admissible functions $u$ to satisfy the constraint $u=0$ on $P_\delta$ and study the asymptotic behaviour of $E_\e$ as $\e$ and $\delta$ go to $0$ according to the interplay between the three scales  $\e$, $\delta$ and $r_\delta$. A natural question is whether or not $E_\e$ and $E_0$ share the same asymptotic behaviour under the imposed constraint on the admissible functions.  We show that this is the case when $\e$ goes to $0$ faster then $r_\delta$. Indeed, the first main result of the paper is Theorem \ref{th:main}, where we show that, if $\e=o(r_\delta)$ and $r_\delta= O(\delta^{d/(d-p)})$, $E_\e$ and $E_0$, subject to the constraint $u=0$ on $P_\delta$,  share the same $\Gamma$-limit, which is given by
\begin{equation}\label{comeloc}
\int_\Omega \left(f_{hom}(\nabla u)+\varphi(u)\right)\, dx, 
\end{equation}
where $\varphi(z)$ is described by the capacitary formula induced by $E_0(u)$, see \eqref{denscapacitaria}. We point out that this kind of homogenization problems in the non local setting has been studied in \cite{PR}, where the authors treat Dirichlet and Neumann boundary conditions for a non local equation in the scalar case and when $p=2$. 

A major difference is when $\e$ scales like $r_\delta$, since in this case the limit functional keeps memories of the non locality of the approximating energies. Indeed,  the second main result of the paper is Theorem \ref{th:mainNL}, where we show that, if $\e=O(r_\delta)$ and $r_\delta= O(\delta^{d/(d-p)})$, the $\Gamma$-limit of $E_\e$, subject to the constraint $u=0$ on $P_\delta$,  is still of the form \eqref{comeloc}, but the density $\varphi(z)$  is now described by a non local capacitary formula, see \eqref{denscapacitariaNL}. If $r_\delta$ does not scale like $\delta^{d/(d-p)}$, in both case $\e=o(r_\delta)$ and $\e=O(r_\delta)$ the asymptotic behaviour of $E_\e$ is trivial and it is consistent with that  of local energies of Dirichlet type in periodically perforated domains (see  Remark \ref{altrescale}). On the contrary, in Theorem \ref{th:sloweps} we show that if $\e\to 0$ slower than $r_{\delta}$, then, for most of the choice of the scaling of $r_{\delta}$ with respect to $\delta$, $E_{\e}$ is not affected by the constraint $u=0$ on $P_\delta$ and thus the $\Gamma$-limit is still given by $E_0$.

From a technical viewpoint, in order to prove Theorem \ref{th:main} and Theorem \ref{th:mainNL} we mainly follow the strategy exploited in \cite{AB}. Nevertheless, the  non local nature of our approximating energies does not allow us to simply adapt that argument to our case. The main difficulty we have encountered in the proof of Theorem \ref{th:main} is to show the convergence of minimum problems on unbounded domains, defining the approximating capacitary densities, to the limit energy density $\varphi$, stated in Proposition \ref{capterm}. A crucial result that allowed us to overcome this difficulty is Theorem \ref{GNS}, which can be considered a non local variant of the  classical Gagliardo-Nirenberg-Sobolev inequality in Sobolev spaces and may be of independent interest and useful for other applications. Specifically,  we show that, fixed $r>0$,  the $L^{p^*}$-norm of suitable piecewise constant interpolations at scale $\e$ of any admissible function is uniformly bounded from above, up to a multiplicative constant, by the energy
$$
G_\e^{r,p}(u,\rr^d)=\int_{B_{r}}\int_\rr^d\left|\frac{u(x+\e\xi)-u(x)}{\e}\right|^p\, dx\, d\xi,
$$
which in turn is controlled by $E_\e(u)$ and plays the role of $\int_{\rr^d}|\nabla u|^p\, dx$ in the Gagliardo-Nirenberg-Sobolev inequality.

We conclude the introduction with some comments about future developments. In our model we refrain from maximal generality in order to emphasize the main features of the asymptotic process for our non local functionals under pinning conditions, but it would be worth extending our analysis to more general integrands. In particular in the critical regime, if one removes the $p$\,-\,homogeneity assumption on $f$, then, assuming the existence of 
$$
\lim_{\delta\to 0} \delta^{dp/(d-p)} f(\xi, \delta^{-d/(d-p)}z)=: f_{\infty,p}(\xi,z)
$$
 the limit energy should be still of the form \eqref{comeloc}, with $f$ replaced by $f_{\infty,p}$ in the definition of the capacitary density $\varphi$. A natural follow-up of our results is also the extension to the case $p=d$ and to the critical regime $\log(r_\delta)= O( \delta^{d/(d-1)})$. We point out that a $\Gamma$-convergence analysis for local functionals in this setting has been provided in \cite{Siga2}. Furthermore, we believe that some of the techniques developed in this paper can also be used to extend the analysis to the case of perforations whose centres are randomly distributed  according to a stationary point process (see \cite{SZZ} for results in this direction in the local case).
We finally point out that another class of non local functionals, namely discrete functionals of the form
$$
\frac{1}{\e^{p+d}}\sum_{i,j\in\mathcal{L}} f(i,j, u_j-u_i),
$$
where $\mathcal{L}$ is a $d$-dimensional lattice, $u:\e\mathcal{L}\to\rr^m$ and $u_i=u(\e i)$, have been widely investigated as a discrete approximation of integral functionals of $p$-growth (see e.g. \cite{alicic}). 
In this setting, an asymptotic analysis similar to the one provided here when $\e=O(r_\delta)$ has been carried on  in \cite{Siga}, where the main result  can be considered the discrete analog of  Theorem \ref{th:mainNL}. It would be interesting to extend that analysis also to the case $\e=o(r_\delta)$.

\section{Notation}\label{setting}

In what follows $d,m\in\NN$ will be fixed natural numbers denoting the dimension of the reference and target spaces of the functions we consider, respectively. Given $t\in\rr$,   $\lfloor t\rfloor$ denotes the integer part of $t$; for $x\in \rr^d$, $r>0$,  $B_r(x)$ (if $x=0$, simply $B_r$) stands for the open ball of centre $x$ and radius $r$, $Q_r(x)$ (if $x=0$, simply $Q_r$) stands for the square  $x+(-\frac{r}{2},\frac{r}{2})^d$. We denote by $S^{d-1}$ the unit sphere in $\rr^d$. If $A$ is a subset of $\rr^d$ then
dist$(x,A)=\inf \{|y-x|: y\in A\}$; 
$\mathcal{A}^{\rm reg}(A)$ is the subfamily of open subsets with Lipschitz boundary. By
 $A\subset \subset B$ we mean that the closure of $A$ is a compact subset of $B$. If $A\subset \subset B$, a {\em cut-off function} between $A$ and $B$ is a (smooth) function $\varphi$ with $0\le\varphi\le 1$, $\varphi=0$ on $\partial B$ and $\varphi=1$ on $A$. Given a real function $h(\cdot)$, we use the symbols $o(h), O(h)$ respectively, to denote  a generic function $g$ such that $\displaystyle\lim\limits_{t\to 0}\frac{g(t)}{h(t)}=0$, $\lim\limits_{t\to 0}\frac{g(t)}{h(t)}=\gamma \in (0, +\infty)$. {If $E$ is a measurable subset of $\rr^d$ we denote by $|E |$ its Lebesgue measure. We use standard notation for Lebesgue and Sobolev spaces. If $u$ is an integrable function on a measurable set $E\subset\rr^d$,
$$u_E:=\frac{1}{|E|}\int_E u(x)dx$$
denotes the average of $u$ on $E$. We use also standard notation for $\Gamma$-convergence \cite{bra2,dal}, indicating the topology with respect to which it is performed. Unless otherwise stated, the letter $C$ denotes a generic strictly positive constant. Relevant dependencies on parameters will be as usual emphasised by putting them in parentheses.

\section{ Setting of the problem and main results}

We fix a growth exponent $p\in(1,d)$ and we let $\Omega\subset\mathbb{R}^d$ be a bounded open set with Lipschitz boundary. Let $\delta, {r}_{\delta}$ be given with $\delta>r_\delta>0$ and set
\begin{equation}\label{perforazioni}
P_\delta:=\bigcup_{i\in\mathbb{Z}^d} B_{r_\delta} (\delta i),
\end{equation}
\begin{equation}\label{lpdelta}
L^p_\delta (\Omega;\mathbb{R}^m):=\{u\in L^p (\Omega;\mathbb{R}^m): u\equiv 0\ \hbox{on}\  P_\delta\cap\Omega\}.
\end{equation}

\noindent Given  $\e>0$ and $f:\mathbb{R}^d\times\mathbb{R}^m\to[0,+\infty)$ a positive Borel function, we introduce the non-local functionals $F_{\e,\delta}:L^p(\Omega;\mathbb{R}^m)\to[0,+\infty]$ defined as
\begin{equation}\label{functionals2}
F_{\e,\delta}(u):=\begin{cases}\displaystyle \int_{\mathbb{R}^d}\int_{\Omega_\e(\xi)} f\Bigl(\xi, D_\e^\xi u(x)\Bigr) dx \, d\xi & \text{if } u\in L^p_\delta (\Omega;\mathbb{R}^m),\cr\cr
+\infty & \text{otherwise},
\end{cases}
\end{equation}
where 
\begin{equation}\label{finitediff}
D_\e^\xi u(x):=\frac{u(x+\e\xi)-u(x)}{\e}
\end{equation}
and for every subset $A$ of $\rr^d$ we set
\begin{equation}\label{insiemiAxi}
A_\e(\xi):=\{x\in A\,|\,x+\e\xi\in A\}.
\end{equation}
We consider the following set of assumptions on the function $f$ above:
\begin{itemize}
\item[($\bf{H}$)]{\bf ($p$\,-\,homogeneity)} $ f(\xi,tz)=t^p f(\xi,z)$ for every $(\xi,z)\in \mathbb{R}^d\times\mathbb{R}^m$ and $t>0$;
\item[($\bf{G}$)] {\bf (growth)} the functions $\displaystyle m(\xi):=\inf_{z\in S^{d-1}} f(\xi,z)$ and $\displaystyle M(\xi):=\sup_{z\in S^{d-1}} f(\xi,z)$ satisfy:
\begin{enumerate}
\item[(G0)] there exist $\lambda_0, r_0>0$ such that $m(\xi)\geq \lambda_0$ if $|\xi|\leq r_0$;
\item[(G1)] $\displaystyle \int_{\mathbb{R}^d} M(\xi)(|\xi|^p+1)\, d\xi<+\infty$;
\end{enumerate}
\item[(\bf{L})]  {\bf ($p$\,-\,Lipschitz continuity)} there exists $C>0$  such that for every $\xi\in\rr^d$ \[ |f(\xi,w)-f(\xi,z)|\leq CM(\xi)(|z|^{p-1}+|w|^{p-1})|w-z| \quad \forall z,w \in\rr^m.  \]
\end{itemize}

\medskip

We also introduce the `truncated'  functionals defined for every $T>0$ as 
\begin{equation}\label{functionals2T}
F_{\e,\delta}^T(u):=\begin{cases}\displaystyle \int_{B_T}\int_{\Omega_\e(\xi)} f\Bigl(\xi, D_\e^\xi u(x)\Bigr) dx \, d\xi & \text{if } u\in L^p_\delta (\Omega;\mathbb{R}^m),\cr\cr
+\infty & \text{otherwise}.
\end{cases}
\end{equation}

Note that $F_{\e,\delta}^T$ is of the form \eqref{functionals2} with $f^T(\xi,z):=\chi_{B_T}(\xi)f(\xi,z)$ in place of $f(\xi,z)$.

\begin{remark}\label{remcrescita}
Note that assumption ({\bf H}) yields that $m(\xi) |z|^p\leq  f(\xi,z)\leq M(\xi) |z|^p$ for every $(\xi,z)\in \mathbb{R}^d\times\mathbb{R}^m$ . 
\end{remark}

\noindent Let us consider also the unconstrained family of functionals $\mathcal {F}_\e: L^p(\Omega;\mathbb{R}^m)\to[0,+\infty]$ defined by
\begin{equation}\label{uncofunctionals}
\mathcal {F}_\e(u):=\int_{\mathbb{R}^d}\int_{\Omega_\e(\xi)} f\Bigl(\xi,D_\e^\xi u(x)\Bigr) dx \, d\xi.
\end{equation}
We also introduce a localized version of such functionals by setting, for any open set $A\subset\rr^d$ and $u\in L^p(A;\rr^m)$,
\begin{equation}\label{uncofunctionalsloc}
\mathcal {F}_\e(u,A):=\int_{\mathbb{R}^d}\int_{A_\e(\xi)} f\Bigl(\xi,D_\e^\xi u(x)\Bigr) dx \, d\xi.
\end{equation}
Moreover the  {\it truncated} functionals are defined, for any $T>0$, as
\begin{equation}\label{uncofunctionalsT}
\mathcal {F}_\e^T(u,A):=\int_{B_T}\int_{A_\e(\xi)} f\Bigl(\xi,D_\e^\xi u(x)\Bigr) dx \, d\xi
\end{equation}
and we drop the dependence on the set  if $A=\Omega$, that is $\mathcal {F}_\e^T(u):=\mathcal {F}_\e^T(u,\Omega)$.

\noindent As a particular case of a more general result, in \cite{AABPT} it was proved the following $\Gamma$-convergence result (see \cite[Theorem 6.1]{AABPT}).

\begin{theorem}\label{convexhom}
Let $\mathcal {F}_\e$ be defined by \eqref{uncofunctionals}, with $f$ satisfying assumptions {\rm({\bf H})} and {\rm ({\bf G})}. Then
\begin{equation}\label{gammaconvex}
\Gamma(L^p)-\lim_{\e\to 0} \mathcal{F}_\e (u)=
\begin{cases}\displaystyle \int_\Omega f_{hom}(\nabla u)\, dx & \text{if}\ u\in W^{1,p}(\Omega;\mathbb{R}^m),\cr\
+\infty & \text{otherwise},
\end{cases}
\end{equation}
where, for $S\in  \mathbb{R}^{m\times d}$,
\begin{equation}\label{homform}
f_{hom}(S):=\lim_{R\to\infty}\frac{1}{R^d}\inf\Big\{\int_{Q_R}\int_{Q_R}f(y-x,v(y)-v(x))dx \, dy : v\in\mathcal{D}^{S}(Q_R)\Big\},
\end{equation}
with
$$
\mathcal{D}^{S}(Q_R) := \big\{u\in L^p(\rr^d;\mathbb{R}^m) : u(x)=Sx \text{ for a.e. } x\in\rr^d\,, \, \dist(x,\rr^d\backslash Q_R)<1 \big\}. 
$$

\end{theorem}


\begin{remark}\label{phomfhom}
Note that the $p$\,-\,homogeneity assumption ({\bf H}) is inherited by $f_{hom}$, that is
$$
f_{hom} (t S)= t^p f_{hom}(S)\ \text{for every}\ t>0\ \text{and}\ S\in\rr^{m\times d}.
$$
Moreover assumption ({\bf G}) easily yields that
$$
m_0 |S|^p\leq f_{hom}(S)\leq M_0|S|^p \ \text{for every}\  S\in\rr^{m\times d}
$$
for two suitable strictly positive constants $m_0, M_0$.
\end{remark}
\begin{remark}\label{convexity} 
Hypothesis ({\bf L}) trivially holds if we assume that $f$ satisfies ({\bf H}), ({\bf G}), and $f(\xi,\cdot)$ is convex for every  $\xi\in\mathbb{R}^d$. Moreover, under this convexity assumption, the asymptotic formula \eqref{homform} reduces to 
$$
f_{hom} (S)=\int_{\mathbb{R}^d} f(\xi, S\xi)\, d\xi
$$ 
(see \cite[Theorem 6.2]{AABPT}). 
\end{remark}

The main results of the paper are provided by the following two theorems.

\begin{theorem}[Local capacitary term]\label{th:main}
Let $F_{\e,\delta}$ be defined by \eqref{functionals2}, with $f$ satisfying assumptions {\rm ({\bf H})}, {\rm ({\bf G})} and {\rm ({\bf L})}. Assume moreover that
\begin{equation}\label{beta}
\lim_{\delta\to 0} \frac{r_\delta}{\delta^{\frac {d}{d-p}}}=\beta
\end{equation}
and that $\delta=\delta_\e$ is such that
\begin{equation}\label{zero}
\lim_{\e\to 0} \frac{\e}{r_{\delta_\e}}=0
\end{equation}
for some real number $\beta>0$.
Then
\[
\Gamma(L^p)\hbox{-}\lim_{\e\to 0} F_{\e,\delta_\e} (u)=
\begin{cases}\displaystyle \int_\Omega f_{hom}(\nabla u)\, dx +\beta^{d-p}\int_\Omega \varphi(u)\, dx & \text{if}\ u\in W^{1,p}(\Omega;\mathbb{R}^m),\cr\
+\infty & \text{otherwise},
\end{cases}
\]
where $f_{hom} (S)$ is defined by \eqref{homform} and for every $z\in \mathbb{R}^m$
\begin{equation}\label{denscapacitaria}
 \varphi(z):=\inf\Big\{\int_{\mathbb{R}^d} f_{hom}(\nabla v)\, dx: \  v-z\in L^{p*}(\mathbb{R}^d;\mathbb{R}^m),\ v\equiv 0\ \text{in}\ B_1,\  v\in W^{1,p}_{\loc}(\mathbb{R}^d;\mathbb{R}^m)\Big\}.
 \end{equation}
\end{theorem} 

\begin{theorem}[Nonlocal capacitary term]\label{th:mainNL}
Let $F_{\e,\delta}$ be defined by \eqref{functionals2}, with $f$ satisfying assumptions {\rm ({\bf H})}, {\rm ({\bf G})} and {\rm ({\bf L})}. Assume moreover that
$$
\lim_{\delta\to 0} \frac{r_\delta}{\delta^{\frac {d}{d-p}}}=\beta
$$
and that $\delta=\delta_\e$ is such that
\begin{equation}\label{alfa}
\lim_{\e\to 0} \frac{\e}{r_{\delta_\e}}=\alpha
\end{equation}
for some real numbers $\alpha,\beta>0$.  
Then
\begin{equation}\label{gammaconvexNL}
\Gamma(L^p)\hbox{-}\lim_{\e\to 0} F_{\e,\delta_\e} (u)=
\begin{cases}\displaystyle \int_\Omega f_{hom}(\nabla u)\, dx +\beta^{d-p}\int_\Omega \varphi_{NL,\alpha}(u)\, dx & \text{if}\ u\in W^{1,p}(\Omega;\mathbb{R}^m),\cr\
+\infty & \text{otherwise},
\end{cases}
\end{equation}
where $f_{hom} (S)$ is defined by \eqref{homform} and for every $z\in \mathbb{R}^m$
 \begin{equation}\label{denscapacitariaNL}
 \varphi_{NL,\alpha}(z):=\inf\left\{\mathcal{F}_\alpha(v,\rr^d):  \ v-z\in L^{p}(\rr^d;\mathbb{R}^m),\ v\equiv 0\ \text{in}\ B_1,\  v-z\  {\rm compactly\  supported}\right\},
 \end{equation}
being  $\mathcal{F}_\alpha$ defined in \eqref{uncofunctionalsloc} with $\e=\alpha$.
 \end{theorem}



\begin{remark}\label{altrescale} 
 Taking into account  the non degeneracy of the capacitary densities \eqref{denscapacitaria}, \eqref{denscapacitariaNL} proved in Proposition \ref{pgrowthdens} below, and 
arguing by comparison, one easily infers that the results stated in Theorem \ref{th:main} and  Theorem \ref{th:mainNL} can be "continuously" extended  to the  case $\beta=0$ and  $\beta=+\infty$. More in details, if $\beta =0$  and either \eqref{zero} or \eqref{alfa} hold,  
then the functionals  $F_{\e,\delta_\e}$ $\Gamma$-\,converge  to the energy functional defined in \eqref{gammaconvex}, while, if $\beta =+\infty$  and either \eqref{zero} or \eqref{alfa} hold,  the $\Gamma$-limit of the functionals $F_{\e,\delta_\e}$  is trivially $0$ if $u\equiv 0$ and $+\infty$ otherwise. This phenomenon is consistent with the asymptotic behaviour  of local energies of Dirichlet type in periodically perforated domains. We will see in Section \ref{sopracritico} that this is not the case if $\alpha=+\infty$. 
\end{remark}

 For later use and reader's convenience we redefine the density functions we have introduced so far in the case $f(\xi,z)$ is replaced by $f^T(\xi,z)=\chi_{B_T}(\xi)f(\xi,z)$. More precisely we set
\begin{equation}\label{homformT}
f_{hom}^T(S):=\lim_{R\to\infty}\frac{1}{R^d}\inf\Big\{\int_{Q_R}\int_{Q_R}f^T(y-x,v(y)-v(x))dx \, dy : v\in\mathcal{D}^{S}(Q_R)\Big\},
\end{equation}
\begin{equation}\label{denscapacitariaT}
 \varphi^T(z):=\inf\Big\{\int_{\mathbb{R}^d} f_{hom}^T(\nabla v)\, dx: \  v-z\in L^{p*}(\mathbb{R}^d;\mathbb{R}^m),\ v\equiv 0\ \text{in}\ B_1,\  v\in W^{1,p}_{\loc}(\mathbb{R}^d;\mathbb{R}^m)\Big\},
 \end{equation}
\begin{equation}\label{denscapacitariaNLT}
 \varphi_{NL,\alpha}^T(z):=\inf\left\{\mathcal{F}_\alpha^T(v,\rr^d):  \ v-z\in L^{p}(\rr^d;\mathbb{R}^m),\ v\equiv 0\ \text{in}\ B_1,\  v-z\  {\rm compactly\  supported}\right\}.
 \end{equation}

\section{Preliminary results}\label{preli}

In this section we collect some results that will be used in Section \ref{supporting}. 
\medskip

\noindent{\bf  Capacity}
\medskip

\noindent Let us recall the notion of $p$-capacity for a given exponent $p\in (1,d)$ (see for instance \cite{EG},\cite{HKM}).  Given an open set $A\subset \rr^d$ and an open  set  $E\subset\subset A$,  the 
{\it relative $p$-capacity} of $E$ in $A$ is defined as 
$$
{\rm cap}_p(E,A)=\inf\left\{ \int_A |\nabla u|^p\, dx : u\in W^{1,p}_0(A), \, u= 1 \hbox{ a.e. in }E\right\}. 
$$
If $A=\rr^d$, we simply write ${\rm cap}_p(E)$. 
It follows by the very definition that the set function ${\rm cap}_p(E,A)$ is increasing in the variable $E$ and decreasing in the variable $A$. 
In addition,  the following properties hold true  
\begin{equation}\label{capmonot}
\begin{split} 
& {\rm cap}_p(E)=\inf\left\{ \int_{\rr^d} |\nabla u|^p\, dx : u\in W^{1,p}_{\rm{loc}} (\rr^d)\cap L^{p^*}(\rr^d), \, u =1 \hbox{ a.e. in }E\right\}\\
&=\lim_{R\to +\infty}{\rm cap}_p(E, B_R)=\inf_{R>0} {\rm cap}_p(E, B_R),
\end{split} 
\end{equation} 
where $p^*:=\frac{pd}{d-p}$ is the conjugate exponent of $p$. It can be also proved that  $ {\rm cap}_p(E)>0$ if $|E|>0$.

 \begin{remark}\label{capavectorial} 
 
 One may also consider, for $z\in \rr^m$,  the vectorial infimum problems 
\begin{equation}\label{capvect}
\inf\left\{ \int_A |\nabla u|^p\, dx : u\in  W^{1,p}_0(A;\rr^m), \, u=z \hbox{ a.e. in }E\right\}.
\end{equation}
Note that, thanks to the $p$\,-\,homogeneity  and the rotational invariance of \eqref{capvect}, it holds
\[
\begin{split} & \inf\left\{ \int_A |\nabla u|^p\, dx : u\in  W^{1,p}_0 (A;\rr^m), \, u=z \hbox{ a.e. in }E\right\}\\
& =|z|^p  \inf\left\{ \int_A |\nabla u|^p\, dx : u\in  W^{1,p}_0 (A;\rr^m), \, u=(1, 0, \ldots, 0) \hbox{ a.e in }E\right\}, 
\end{split} 
\]
and the infimum in the  last term can be in turn confined to functions $v\in W^{1,p}_0(A;\rr^m) $ such that $v=(v_1,0, \ldots, 0)$.
Hence, we can conclude that, for any $z\in \rr^m$,  
\begin{equation}\label{stimaciuno}
\inf\left\{ \int_A |\nabla u|^p\, dx : u\in W^{1,p}_0(A; \rr^m), \, u= z \hbox{ a.e. in }E\right\} = {\rm cap}_p(E,A) |z|^p .
\end{equation}
\end{remark}
\medskip

\noindent{\bf Convolution-type energies}
\medskip

\noindent The following results, contained in \cite{AABPT},   extend corresponding results  in Sobolev spaces to the case of convolution-type energies.

\noindent Let $r,\e>0$, and $p>1$. We set, for every  open set $A \subset \rr^d$ and $u\in L^p(A;\rr^m)$, 
\begin{equation}\label{Gfunctional}
G_\e^{r,p}(u,A)=\int_{B_{r}}\int_{A_\e(\xi)}\left| D_\e^\xi u(x)\right|^p\, dx\, d\xi,
\end{equation}
where $D_\e^\xi u(x)$ and $A_\e(\xi)$ are defined by \eqref{finitediff} and \eqref{insiemiAxi}, respectively.

\noindent The next proposition rephrases Lemma 4.1 in  \cite{AABPT} where the authors show that long-range energy contributions can be controlled by the short-range energy $G_\e^{r,p}$. 

\begin{proposition}\label{boundlemma}
For every $r>0$ there exists a positive constant $C$ such that, for any open set $E\subset\Omega$, 
for every $\xi\in\mathbb{R}^d$
and $u\in L^p(\Omega;\mathbb{R}^m)$,
there holds
$$\int_E\Big|\frac{u(x+\e\xi)-u(x)}{\e}\Big|^p dx \le C(|\xi|^p+1) G_\e^{r,p}(u, E+B_{\e(r+|\xi|)}). $$
 for any $\e>0$ such that 
\begin{equation}\label{eps-small}
\e \,r < \dist(E+B_{\e(r+|\xi|)},\Omega^c).
\end{equation}
\end{proposition}
 
 \begin{remark}\label{tuttolospazio}
Note that if $\Omega=\rr^d$ then \eqref{eps-small} is satisfied for any $\e>0$ and $\xi\in\rr^d$, deriving in particular the following estimate, which will be useful later: for every $r>0$, there exists $C=C(r)$ such that, for every $\xi\in\mathbb{R}^d$, $\e>0$, and $u\in L^p_{loc}(\rr^d;\rr^m)$, with $u\equiv z$ on $\rr^d\setminus B_{R-r\e}$, $z\in\rr^m$, there holds
\begin{equation}\label{eqtuttoerred}
\int_{\rr^d}\Big|\frac{u(x+\e\xi)-u(x)}{\e}\Big|^p dx \le C(|\xi|^p+1) G_\e^{r,p}(u,\rr^d).
\end{equation}
\end{remark}

\noindent As a consequence of Lemma \ref{boundlemma} and a result concerning  extension operators (see \cite[Theorem 4.1]{AABPT}),  the following estimate is derived. 
\begin{corollary}\label{boundlemma-lip}[\cite[Corollary 4.1]{AABPT}]
For any open set $A\in\mathcal{A}^{\rm reg}(\Omega)$ and $r>0$ there exist two positive constants $C=C(A)$ and $\e_0=\e_0(r,A)$ such that for every $\xi\in\mathbb{R}^d$ and $u\in L^p(A;\mathbb{R}^m)$ there holds
$$\int_{A_\e(\xi)}\left|\frac{u(x+\e\xi)-u(x)}{\e}\right|^p dx \le C(|\xi|^p+1) \big( G_\e^{r,p}(u,A)+\|u\|_{L^p(A;\rr^m)}^p \big),$$
for every $\e<\e_0$.
\end{corollary}
The following theorem states the analogue of the classical Poincar\'e-Wirtinger inequality for the functionals $G_\e^{r,p}$.
\begin{theorem}\label{pwineq}[\cite[Proposition 4.2]{AABPT}]
Let $r>0$ and let $A$ be a bounded connected open set of $\rr^d$ with Lipschitz boundary. Then for every measurable set $E\subset A$ with $|E|>0$ there exists a positive constant $C=C(A,E)$ such that for any $u\in L^p(A;\mathbb{R}^m)$ and  $\e>0$
\begin{equation}\label{dispw}
\int_A |u(x)-u_E|^p dx\le C G^{r,p}_\e(u,A).
\end{equation}
\end{theorem}

\noindent When we replace $E$ and $A$ with a translation of $\lambda E$ and $\lambda A$, respectively, being $\lambda>0$ a scaling factor, we have the following result. 
\begin{proposition}\label{pwscaling}
Let $r>0$ and let $A$ be a bounded connected open set of $\rr^d$ with Lipschitz boundary. Then for every measurable set $E\subset A$ with $|E|>0$ there exists a positive constant $C=C(A,E)$ such that  for every  $x_0\in \rr^d$, $u\in L^p(\lambda A+x_0;\mathbb{R}^m)$ and  $\e>0$ 
\begin{equation}\label{dispw2}
\int_{\lambda A+x_0} |u(x)-u_{\lambda E+x_0}|^p dx\le C\,\lambda^p \,G^{r,p}_\e(u,\lambda  A+x_0).
\end{equation}
\end{proposition}
\begin{proof} It is not restrictive to assume that $x_0=0$. 
If $u\in L^p(\lambda A;\mathbb{R}^m)$ the function $w(y)=u(\lambda y)$ belongs to $ L^p(A;\mathbb{R}^m)$. Writing inequality \eqref{dispw} for $w$ with $\e$ replaced by $\frac{\e}\lambda$ we get
$$
\int_{A} |w(y)-w_{E}|^p dy\le C(A,E) \,G^{r,p}_{\frac{\e}{\lambda}}(w,A).
$$
On the other hand $w_{E}=u_{\lambda E}$ and the change of variable $x=\lambda y$ gives the desired result.
\end{proof}
Eventually, the next result accounts  for the compactness in the strong $L^p$-topology of sequences of functions with uniformly bounded energy on a regular bounded set of $\rr^d$. 


\begin{theorem}\label{kolcom}[\cite[Theorem 4.2]{AABPT} ]
Let $A$ be any bounded  open Lipschitz set of $\rr^d$ and let $\{u_\e\}_\e\subset L^p(A;\mathbb{R}^m)$ be such that for some $r>0$
$$\sup_{\e>0} \left\{\| u_\e\|_{L^p(A;\rr^m)}+G_\e^{r,p}(u_\e,A)\right\}<+\infty.$$
Then, for any  $\e_j\to 0$, $\{u_{\e_j}\}_j$ is relatively compact in $L^p(A;\mathbb{R}^m)$ and every limit of a converging subsequence lies  in $W^{1,p}(A;\mathbb{R}^m)$.
\end{theorem}

\section{Gagliardo-Nirenberg-Sobolev type inequality}
In this section we state and prove a crucial result for our analysis which may be of independent interest and resembles the classical Gagliardo-Nirenberg-Sobolev inequality in Sobolev spaces. Its  proof follows the lines of the proof of the corresponding result in Sobolev spaces. Such a result will allow us to prove the convergence of the infimum problems defining the approximating capacitary energy densities (see Proposition \ref{capterm}). 

\noindent For $r, \sigma>0$, set
$$
PC_\sigma(\rr^d;\rr^m):=\{u:\rr^d\to\rr^m:\ u \ \mbox{is constant on} \ \sigma k+[0,\sigma)^d\ \forall\ k\in\ZZ^d\}
$$
and, given $p\ge 1$, let $T_\e: L^p(\rr^d;\rr^m)\to PC_{\tilde{r}\e}(\rr^d;\rr^m)$ be defined by
\begin{equation}\label{intcost}
T_\e u(x):=\frac{1}{(\tilde{r}\e)^d}\int_{\tilde{r}\e k+[0,\tilde{r}\e)^d} u(y)\, dy\quad \mbox{on}\ \tilde{r}\e k+[0,\tilde{r}\e)^d,\ k\in\ZZ^d, 
\end{equation}
where
$$
{\tilde r}:=\frac{r}{\sqrt{d+3}}. 
$$ 
In the following result we extend the definition of $G_\e^{r,p}((u,\rr^d)$ in \eqref{Gfunctional} for $p=1$. 

\begin{theorem}[\bf Gagliardo-Nirenberg-Sobolev type inequality]\label{GNS}
Let $p\in[1,d)$. Then there exists a constant $C=C(p,d, r)>0$ such that for every $u\in L^p(\rr^d;\rr^m)$
\begin{equation}\label{GNSineq}
\left(\int_{\rr^d} |T_\e u(x)|^{p^*}\, dx\right)^{\frac{p}{p^*}}\leq C G^{r,p}_\e(u,\rr^d),
\end{equation}
where $p^\ast:=\frac{pd}{d-p}$.
\end{theorem}
\begin{proof}
By a density argument and the  $L^p$-continuity of $G_\e^{r,p}(\cdot,\rr^d)$, it is enough to prove the inequality  \eqref{GNSineq} for $u\in L^p(\rr^d;\rr^m)$ with compact support.

\noindent Let us first consider the case $p=1$ and fix such a $u$.  
We introduce some notation. For $k=(k_1,\dots,k_d)\in\ZZ^d$, set
\[
Q_{k}^\varepsilon:=\tilde{r}\e k+\left[0,\tilde{r}\varepsilon\right)^d
\]
and, for every $j=1,\dots,d$,  let $\hat k_j\in\ZZ^{d-1}$ defined by $\hat k_j=(k_1,\dots,k_{j-1},k_{j+1},\dots, k_d)$. Moreover, with fixed $k\in\ZZ^d$ and $j=1,\dots,d$, set $\hat k_j(h):=(k_1,\dots,k_{j-1},h,k_{j+1},\dots,k_d)$ for every $h\in\ZZ$, 
and denote by $x_{h}$ an independent variable lying in the cube $Q_{\hat k_j(h)}^\e$. In particular, with the notation above, we may write $ u(x_{k_1})=\sum\limits_{h=-\infty}^{k_1}(u(x_{h})-u(x_{h-1})),$
\ being actually the latter a finite sum by the compactness of the support of $u$. Thus integrating in all the variables  $x_{h}$ with $h\leq k_1$, we get
\[
\left|\int_{Q_k^\varepsilon}u(x_{k_1})\,dx_{k_1}\right|\le \frac{1}{(\tilde{r}\e)^d}\sum_{h=-\infty}^{k_1}\int_{Q_{\hat k_1(h)}^\e}\int_{Q_{\hat k_1(h-1)}^\e}|u(x_{h})-u(x_{h-1})|\,dx_{h-1}\,dx_{h},
\]
and, using the definition of $T_\e$ given in \eqref{intcost}, 
\[
\e^{d-1}|T_\e u(\e k)|\le C\frac{1}{\e^d}\sum_{h=-\infty}^{k_1}\int_{Q_{\hat k_1(h)}^\e}\int_{Q_{\hat k_1(h)}^\e}\left|\frac{u(x_{h})-u(x_{h-1})}{\e}\right|\,dx_{h-1}\,dx_{h},
\]
where $C=C(r,d)$. 
For every $j=1,\cdots,d$, we define the stripe $S^\e_{{\hat k}_j}=\{(x_1,\cdots, x_d)\in\rr^d :  {\tilde r}_0\e k_i\le x_i<{\tilde r}_0\e (k_i+1), \forall i\neq j\}$, thus we have 
\[
\e^{d-1}|T_\e u(\e k)|\le C\,G_\e^{r,1}(u,S^\e_{{\hat k}_1}).
\]
Analogously, for $j=2,\cdots,d$,
\[
\e^{d-1}|T_\e u(\e k)|\le C\,G_\e^{r,1}(u,S^\e_{{\hat k}_j}),
\]
which in turn implies
\[
\e^{d(d-1)}|T_\e u(\e k)|^d\le C\prod_{j=1}^d G_\e^{r,1}(u,S^\e_{{\hat k}_j}),
\]
or, equivalently,
\[
\e^{d}|T_\e u(\e k)|^{\frac{d}{d-1}}\le C\prod_{j=1}^d \left(G_\e^{r,1}(u,S^\e_{{\hat k}_j})\right)^{\frac{1}{d-1}}.
\]
Summing over $k_1\in\ZZ$ and using H\"older's inequality, we get
\[
\begin{split}
\sum_{k_1\in\ZZ}\e^{d}|T_\e u(\e k)|^{\frac{d}{d-1}}&\le  C\left(G_\e^{r,1}(u,S^\e_{{\hat k}_1})\right)^{\frac{1}{d-1}}\sum_{k_1\in\ZZ}\prod_{j=2}^d \left(G_\e^{r,1}(u,S^\e_{{\hat k}_j})\right)^{\frac{1}{d-1}}\\
&\le C\left(G_\e^{r,1}(u,S^\e_{{\hat k}_1})\right)^{\frac{1}{d-1}}\prod_{j=2}^d\left(\sum_{k_1\in\ZZ}G_\e^{r,1}(u,S^\e_{{\hat k}_j})\right)^{\frac{1}{d-1}}.
\end{split}
\]
 Next we sum over $k_2\in\ZZ$ and we use again H\"older's inequality, obtaining
\[
\begin{split}
&\sum_{k_2\in\ZZ}\sum_{k_1\in\ZZ}\e^{d}|T_\e u(\e k)|^{\frac{d}{d-1}}\\
&\le C\left(\sum_{k_1\in\ZZ}G_\e^{r,1}(u,S^\e_{{\hat k}_2})\right)^{\frac{1}{d-1}}\sum_{k_2\in\ZZ}\left(G_\e^{r,1}(u,S^\e_{{\hat k}_1})\prod_{j=3}^d\sum_{k_1\in\ZZ}G_\e^{r,1}(u,S^\e_{{\hat k}_j})\right)^{\frac{1}{d-1}}\\
&\le C \left(\sum_{k_1\in\ZZ}G_\e^{r,1}(u,S^\e_{{\hat k}_2})\right)^{\frac{1}{d-1}}\left(\sum_{k_2\in\ZZ}G_\e^{r,1}(u,S^\e_{{\hat k}_1})\right)^{\frac{1}{d-1}}\prod_{j=3}^d\left(\sum_{k_2\in\ZZ}\sum_{k_1\in\ZZ}G_\e^{r,1}(u,S^\e_{{\hat k}_j})\right)^{\frac{1}{d-1}}.
\end{split}
\]
We iterate the procedure, finding out
\[
\begin{split}
\sum_{i=1}^d\sum_{\substack{k_i\in\ZZ}}\e^d|T_\e u(\e k)|^{\frac{d}{d-1}}\le  C\prod_{j=1}^d\left(\sum_{\substack{ i=1\\i\neq j}}^d\sum_{\substack{k_i\in\ZZ }}G_\e^{r,1}(u,S^\e_{{\hat k}_j})\right)^{\frac{1}{d-1}}\le C\left(G_\e^{r,1}(u,\rr^d)\right)^{\frac{d}{d-1}}.
\end{split}
\]
On the other hand, we have 
\[
\int_{\rr^d}|T_\e u(y)|^{1^*}dy=\sum_{k\in\ZZ^d}\int_{Q_k^\e}|T_\e u(y)|^{1^*}dy=\sum_{i=1}^d\sum_{\substack{k_i\in\ZZ}}(\tilde{r}\e)^d|T_\e u(\e k)|^{\frac{d}{d-1}};
\]
therefore
\begin{equation}\label{sgnp1}
\left(\int_{\rr^d}|T_\e u(y)|^{1^*}dy\right)^{\frac{1}{1^*}}\le C\,G_\e^{r,1}(u,\rr^d).
\end{equation}
Thus the theorem is proven for $p=1$. 
\smallskip

If $1<p<d$, given $u\in L^p((\rr^d;\rr^m)$ with compact support, we use  \eqref{sgnp1} with $|T_\e u|^\gamma$  in place of $u$, for some $\gamma>1$ to be chosen later. We obtain
\[
\begin{split}
&\left(\int_{\rr^d}|T_\e u(y)|^{\gamma1^*}dy\right)^{\frac{1}{1^*}}\le C\int_{B_{r}}\int_{\rr^d}\left|\frac{|T_\e u(x+\e\xi)|^\gamma-|T_\e u(x)|^\gamma}{\e}\right|\,dx\,d\xi\\
&\le C\int_{B_{r}}\int_{\rr^d}(|T_\e u(x+\e\xi)|^{\gamma-1}+|T_\e u(x)|^{\gamma-1})\left|\frac{T_\e u(x+\e\xi)-T_\e u(x)}{\e}\right|\,dx\,d\xi, 
\end{split}
\]
where now $C$ depends on $\gamma,d,r$. By using H\"older's inequality with $p$ and $p'=\frac{p}{p-1}$, we also get
\[
\begin{split}
\left(\int_{\rr^d}|T_\e u(y)|^{\gamma1^*}dy\right)^{\frac{1}{1^*}}\le C&\left(\int_{B_{r}}\int_{\rr^d}(|T_\e u(x+\e\xi)|^{(\gamma-1)p'}+|T_\e u(x)|^{(\gamma-1)p'})\,dx\,d\xi\right)^{\frac{1}{p'}}\\
&\quad \times G_\e^{r,p}(T_\e u,\rr^d)^{\frac{1}{p}}\le C\left(\int_{\rr^d}|T_\e u(x)|^{(\gamma-1)p'}dx\right)^{\frac{1}{p'}}G_\e^{r,p}(T_\e u,\rr^d)^{\frac{1}{p}},
\end{split}
\]
possibly for a different constant $C=C(\gamma,p,d,r)$.
Choose $\gamma$ so that $\gamma 1^*=\gamma\frac{d}{d-1}=(\gamma-1)p'$, and accordingly $\gamma=\frac{d-1}{d}p^*$. Thus
\[
\left(\int_{\rr^d}|T_\e u(y)|^{p^*}dy\right)^{\frac{p}{p^*}}\le C\,G_\e^{r,p}(T_\e u,\rr^d).
\]
We now observe that
\[
\begin{split}
G_\e^{r,p}(T_\e u,\rr^d)&\le C\int_{B_{r}}\int_{\rr^d}\left|\frac{T_\e u(x)-u(x)}{\e}\right|^pdx\,d\xi+C\int_{B_{r}}\int_{\rr^d}\left|\frac{T_\e u(x+\e\xi)-u(x+\e\xi)}{\e}\right|^p dx\,d\xi\\
&+C\int_{B_{r}}\int_{\rr^d}\left|\frac{ u(x+\e\xi)-u(x)}{\e}\right|^p dx\,d\xi,
\end{split}
\]
the last term being, up to a multiplicative constant, the functional $G_\e^{r,p}(u,\rr^d)$. 
The first and the second term on the right hand side may be estimated as follows
\[
\begin{split}
\int_{B_{r}}\int_{\rr^d}\left|\frac{T_\e u(x)-u(x)}{\e}\right|^pdx\,d\xi&=|B_{r}|\int_{\rr^d}\left|\frac{T_\e u(x)-u(x)}{\e}\right|^pdx\\
&=|B_{r}|\sum_{k\in\ZZ^d}\int_{Q^\e_{k}}\left|\mint_{Q^\e_{k}}\frac{u(y)-u(x)}{\e}\,dy\right|^pdx\\
&\le |B_{r}|\sum_{k\in\ZZ^d}\int_{Q^\e_{k}}\mint_{Q^\e_{k}}\left|\frac{u(y)-u(x)}{\e}\right|^pdy\,dx\\
&=C \sum_{k\in\ZZ^d}\frac{1}{\e^d}\int_{Q^\e_{k}}\int_{Q^\e_{k}}\left|\frac{u(y)-u(x)}{\e}\right|^pdy\,dx\le C\,G_\e^{r,p}(u,\rr^d),
\end{split}
\]
where we have used Jensen's inequality.  This concludes the proof.
\end{proof}
\noindent By Theorem \ref{GNS} and Theorem \ref{kolcom}, we deduce the following compactness result.

\begin{corollary}\label{compactnessloc}

Let $p\in (1,d)$ and let $\{u_\e\}_\e\subset L^p(\rr^d;\mathbb{R}^m)$ be such that for some $r>0$

\begin{itemize}
\item[(i)] $\displaystyle\sup_{\e>0}\| u_\e\|_{L^p(K;\rr^m)}<+\infty$ for every compact set $K\subset \rr^d$;
\item[(ii)]$\displaystyle\sup_{\e>0}\, G_\e^{r,p}(u_\e,\rr^d)<+\infty$.

\end{itemize}
Then, for any  $\e_j\to 0$, $\{u_{\e_j}\}_j$ is relatively compact in $L^p_{\loc}(\rr^d;\mathbb{R}^m)$ and every limit of a converging subsequence lies  in $W^{1,p}_{\loc}(\rr^d;\mathbb{R}^m)\cap 
L^{p^*}(\rr^d;\rr^m)$.

\end{corollary}
\begin{proof}
Theorem \ref{kolcom} yields that for any bounded  open Lipschitz set  $A\subset \rr^d$   $\{u_{\e_j}\}_j$ is relatively compact  in $L^p(A;\rr^m)$ and any limit point lies in $W^{1, p}(A;\rr^m)$.  By a standard diagonalization argument,    $\{u_j\}_j$  is also relatively compact in $L^p_{\loc} (\rr^d; \rr^m)$ and any limit point lies in $ W^{1,p}_{\loc}(\rr^d;\mathbb{R}^m)$. Let us consider a subsequence (not relabelled)  $u_{\e_j}$ and $u\in W^{1,p}_{\loc}(\rr^d;\mathbb{R}^m)$ such that $u_j\to u$ strongly in $L^p_{\loc} (\rr^d; \rr^m)$ and pointwise in $\rr^d$. Then it can be proved that also $T_{\e_j} u_{\e_j}\to u$  strongly  in $L^p_{\loc} (\rr^d; \rr^m)$ and pointwise in $\rr^d$ (see also \cite[Lemma 2.11]{AFG} ). Thus, it is enough  to apply  Fatou's Lemma in \eqref{GNSineq} and use  hypothesis $(ii)$ to deduce that $u\in L^{p^*}(\rr^d;\rr^m)$.
\end{proof}

\section{Supporting results}\label{supporting}

In this section we present some key results of technical flavor that we are later going to use for the proof of Theorem \ref{th:main} and Theorem \ref{th:mainNL}.


\subsection{A joining lemma}\label{join}

Here we state and prove the analog of Lemma 3.1 in \cite{AB} for our non local  functionals and we follow the lines of its proof. It allows to restrict our attention in our $\Gamma$-convergence analysis to sequences of converging functions that are constant on suitable annuli surrounding the perforations.

\begin{lemma}\label{raccordo}
Let $\delta_j\to 0$ as $j\to +\infty$. Let $T\geq r_0$ be fixed and let $0<\varepsilon_j<\rho_j<\frac{\delta_j}{2}$ with $\e_j=o(\rho_j)$ as $j\to +\infty$. Let  $u_j$ converge to $u$ in $L^p(\Omega,\mathbb{R}^m)$ with $\sup_j \mathcal {F}_{\e_j}^T(u_j)<+\infty$.
Set
\[
Z_j(\Omega):=\{i\in\mathbb{Z}^d : \dist(i\delta_j,\mathbb{R}^d\setminus\Omega)>\delta_j\}
\]
and, for $h\in\NN$ and $i\in Z_j(\Omega)$, set
\begin{equation}\label{anelli}
A_j^{i,h}:=\{x\in\Omega : 2^{-h-1}\rho_j<|x-i\delta_j|<2^{-h}\rho_j\},
\end{equation}
\begin{equation}\label{medianelli}
u_j^{i,h}:=\mint_{A_j^{i,h}}u_j\ \ \ \ \ \ \rho_{j,h}:=\frac{3}{4}2^{-h}\rho_j.
\end{equation}
 Then, given  $N\in\NN$, for every $i\in Z_j(\Omega)$ there exists $k_i\in\{0,\cdots, N-1\}$ and a sequence $w_j$ still converging to $u$ in $L^p(\Omega,\rr^m)$, such that for $j$ sufficiently large
\begin{eqnarray}
&w_j=u_j \hbox{ on }\Omega\setminus\cup_{i\in Z_j(\Omega)}A_j^{i,k_i},\label{wjdef}\\
&w_j=u_j^{i,k_i} \hbox{ on }{\partial} B_{\rho_{j,k_i}}(i\delta_j)+B_{T \e_j},\label{wjdef2}\\
&\left |\mathcal {F}_{\e_j}^T(w_j)-\mathcal {F}_{\e_j}^T(u_j)\right|\le\displaystyle \frac{C}{N}.\label{energie}
\end{eqnarray}

\end{lemma}

\begin{proof}
Notice that, by (\textbf{H}) and (G0), $\sup_j G^{r_0,p}_{\e_j}(u_j,\Omega)<+\infty$. Let $\varphi=\varphi_j^{i,h}\in C^\infty_c(A_j^{i,h})$ be such that $\varphi=1$ on $\partial B_{\rho_{j,k_i}}(i\delta_j)+B_{T \e_j}$ and $|D\varphi|\le \frac{C}{\rho_{j,h}}$ and define the function $w_j^{i,h}=\varphi u_j^{i,h}+(1-\varphi)u_j$. Adding and subtracting the quantity $\varphi(x+\e_j\xi) u_j(x)$ in the argument of $f$, and using (G1) and the convexity of the power function, we have

\[
\begin{split}
&\mathcal {F}_{\e_j}^T(w_j^{i,h}, A_j^{i,h})=\\
&=\int_{B_T}\int_{(A_j^{i,h})_{\e_j}(\xi)}f\left(\xi,\frac{(u_j^{i,h}\varphi+(1-\varphi)u_j)(x+\e_j\xi)-(u_j^{i,h}\varphi+(1-\varphi)u_j)(x)}{\e_j}\right)dx\,d\xi\\
&\le C \int_{B_T}M(\xi)\int_{(A_j^{i,h})_{\e_j}(\xi)}\left|\frac{u_j(x+\e_j\xi)-u_j(x)}{\e_j}\right|^p+\left |u_j^{i,h}-u_j(x)\right|^p\left|\frac{\varphi(x+\e_j\xi)-\varphi(x)}{\e_j}\right|^pdx\,d\xi.
\end{split}
\]
Recalling that $|D\varphi|\le \frac{C}{\rho_{j,h}}$ and using the Poincar\'e-Wirtinger  inequality  in Proposition \ref{pwscaling} with $x_0=i \delta_j$ and $\lambda = \frac{4}{3}\rho_{j,h}$, we have
\[
\begin{split}
&\int_{B_T}M(\xi)\int_{(A_j^{i,h})_{\e_j}(\xi)}\left |u_j^{i,h}-u_j(x)\right |^p\left|\frac{\varphi(x+\e_j\xi)-\varphi(x)}{\e_j}\right|^pdx\,d\xi\\
&\le \frac{C}{\rho_{j,h}^{\, p}}\int_{B_T}M(\xi)|\xi|^p\,d\xi\int_{A_j^{i,h}}\left |u_j^{i,h}-u_j(x)\right|^p\,dx\le C\,{G}^{r_0}_{\e_j}(u_j, A_j^{i,h}).
\end{split}
\]
On the other hand, taking into account that $\e_j=o(\rho_j)$ we deduce that,  for $j$ sufficiently large, it holds    
$$A_j^{i,h} +B_{\e_j(r_0+T)} \subset \bigcup_{\ell=h-1, h, h+1} A_j^{i,\ell}=:{\tilde A}_j^{i,h},  \quad h=0,\cdots ,N-1.
$$
In addition, for $j$ sufficiently large, thanks to the fact that $\dist({\tilde A}_j^{i,h},\Omega\setminus B_{\rho_j}(i\delta_j) )\sim \rho_j$, we also have 
$$\e_j<r_0^{-1}\dist({\tilde A}_j^{i,h},\Omega\setminus B_{\rho_j}(i\delta_j) ).$$
Therefore,  using  Proposition \ref{boundlemma}, we get 
\[
\begin{split}
 \int_{B_T}M(\xi)&\int_{(A_j^{i,h})_{\e_j}(\xi)}\left|\frac{u_j(x+\e_j\xi)-u_j(x)}{\e_j}\right|^p\,dx\,d\xi\\
 &\le C\,\int_{B_T}M(\xi)(|\xi|^p+1)\,d\xi G^{r_0,p}_{\e_j}(u_j,{\tilde A}_j^{i,h}) \\
 &\le C\,G^{r_0,p}_{\e_j}(u_j,{\tilde A}_j^{i,h}).
\end{split}
\]

\noindent Hence, the previous estimates yield that
\begin{equation}\label{stimanello}
\mathcal {F}_{\e_j}^T(w_j^{i,h}, A_j^{i,h})\le C\,G^{r_0,p}_{\e_j}(u_j,{\tilde A}_j^{i,h}).
\end{equation}
Since the sets ${\tilde A}_j^{i,h}$ overlap at most $3$ times,  with fixed $N\in\NN$,   we sum over $h=0,\cdots ,N-1$ and get 
\[
\sum_{h=0}^{N-1}G^{r_0,p}_{\e_j}(u_j,{\tilde A}_j^{i,h})\le 3 \, G^{r_0,p}_{\e_j}(u_j, B_{\rho_j}(i\delta_j)).
\]
Hence there exists $k_i\in\{0,\cdots, N-1\}$ such that
\[
{G}^{r_0,p}_{\e_j}(u_j, {\tilde A}_j^{i,k_i})\le \frac{3}{N} G^{r_0,p}_{\e_j}(u_j, B_{\rho_j}(i\delta_j)),
\]
which in turn yields
\[
\mathcal{F}_{\e_j}^T(w_j^{i,k_i},  A_j^{i,k_i})\le \frac{C}{N} G^{r_0,p}_{\e_j}(u_j, B_{\rho_j}(i\delta_j)).
\]
Noticing that estimate \eqref{stimanello} holds even if we replace $w_j^{i,h}$ with $u_j$, we get
\[
\begin{split}
\left |\mathcal{F}_{\e_j}^T(u_j, A_j^{i,k_i})-\mathcal{F}_{\e_j}^T(w_j^{i,k_i}, A_j^{i,k_i})\right|&\le \mathcal{F}_{\e_j}^T(u_j, A_j^{i,k_i})+\mathcal{F}_{\e_j}^T(w_j^{i,k_i}, A_j^{i,k_i})\\
&\le  \frac{C}{N} G^{r_0,p}_{\e_j}(u_j, B_{\rho_j}(i\delta_j)).
\end{split}
\]
\noindent Then \eqref{wjdef}, \eqref{wjdef2}\eqref{energie} are satisfied by $w_j$ defined as
\[
w_j(x):=
\begin{cases}u_j (x) &\text {if } x\in \Omega\setminus \cup_{i\in Z_j(\Omega)}A_j^{i,k_i},\cr\cr
\varphi_j^{i,k_i}(x) u_j^{i,k_i}(x)+(1-\varphi_j^{i,k_i}(x))u_j(x) & \text{if } x\in A_j^{i,k_i},\ i\in Z_j(\Omega).
\end{cases}
\]
\noindent We finally prove the convergence of $w_j$ to $u$ in $L^p(\Omega;\rr^m)$. We have
\[
\begin{split}
\int_\Omega \left|w_j-u_j\right|^p\,dx&=\sum_{i\in Z_j(\Omega)}\int_{ A_j^{i,k_i}}\left|\varphi_j^{i,k_i} u_j^{i,k_i}+(1-\varphi_j^{i,k_i})u_j-u_j\right|^p\,dx\\
&\le \,\sum_{i\in Z_j(\Omega)}\int_{A_j^{i,k_i}}|u_j^{i,k_i}-u_j|^p\,dx
\le C\,\sum_{i\in Z_j(\Omega)}(\rho_j^{k_i})^pG^{r_0,p}_{\e_j}(u_j,A_j^{i,k_i})\\
&\le C\,\rho_j^p\sum_{i\in Z_j(\Omega)}G^{r_0,p}_{\e_j}(u_j,A_j^{i,k_i})\le C\,\rho_j^p,
\end{split}
\]
where we used again Proposition \ref{pwscaling}  in the second line.
Hence, passing to the limit as $j$ tends to $+\infty$ we get the desired convergence.

\end{proof}
\subsection{Truncation Lemma}
By the composition with a  suitable lipschitz function, the following technical lemma allows us to replace  a given sequence $u_j$ with equibounded energies and $L^p$-norms, by  a new sequence  uniformly bounded in $L^\infty$ and with a small gap  in energy. The proof is strongly inspired by  Lemma 3.5 in \cite{BDFV}.

\begin{lemma}\label{truncation}
Let $\{u_j\}$  with $\sup_j(\mathcal{F}_{\varepsilon_j}(u_j) +\|u_j\|_{L^p(\Omega; \rr^m)})<+\infty$. Then for every $\eta > 0$ and 
$M> 1$ there exist $R_M>M>0$ and  a sequence of Lipschitz functions $\Phi_{j,M}: \rr^m\to\rr^m$ with $Lip(\Phi_{j,M})= 1$, $\Phi_{j,M} (z) = z$ if $|z| < M$ and $\Phi_{j,M}(z) = 0$ if $|z| > R_M$, such that
 it holds 
\[
\mathcal{F}_{\e_j}(\Phi_{j,M}(u_j))\le \mathcal{F}_{\e_j}(u_j)+\eta
\]
 for every $j\in \NN$ such that $\e_j<\e_0$, with $\e_0$ depending on $\Omega$. Moreover we can extract a subsequence $(j_k)$  such that $\Phi_{j_k,M}=:\Phi_M$ do not depend on $k\in\NN$.
\end{lemma}

\begin{proof}
Note that, by assumption (G0), $G^{r_0,p}_{\e_j}(u_j,\Omega)$ is uniformly bounded.  Set 
\begin{equation}\label{bound1}
C_1:=\sup_j (G^{r_0,p}_{\e_j}(u_j,\Omega)+\|u_j\|_{L^p(\Omega; \rr^m)}),
\end{equation}
\begin{equation}\label{bound2}
C_2= 6\, C(\Omega, r_0) \int_{\rr^d}M(\xi)(|\xi|^p+1)\,d\xi , 
\end{equation}
where $C(\Omega, r_0)$ is the constant obtained by Corollary \ref{boundlemma-lip} applied 
with $A=\Omega$ and $r=r_0$. 

Let $\eta>0$ and $M> 0$ be fixed. Note that,  once the statement is proved for a given positive constant $M$, then it holds true also for any $M'<M$, hence up  to replace $M$ with a bigger value  it is not restrictive to assume that $M$ is an integer and satisfies 
\begin{equation}\label{bound3} 
M>\lfloor \frac{2 \,C_1C_2}{\eta}\rfloor+2.
\end{equation}
For $h=1, \ldots, M$ let $\Phi_{M}^h: {\rr}^m\to{\rr}^m$ be a Lipschitz function  such that 
\[
\Phi_{M}^h (z) =
\left\{ 
\begin{array}{ll}
z &\hbox{ if } |z|\le M^h\\
0 &\hbox{ if } |z| > M^{h+1}.\\
\end{array}
\right.
\]
and $\Phi^h_{M}$ connects linearly in the radial directions the values on the boundary of the annulus $\{z\in {\rr}^m:  M^{h}< |z| < M^{h+1}\}$. 
A quick computation shows that for any $h=1, \ldots, M$ $Lip (\Phi_{M}^h)\le \frac{1}{M-1}<1$ on the annulus, thus  $Lip(\Phi_{M}^h)= 1$. Let $w_j^h=\Phi_{M}^h(u_j)$ and estimate $\mathcal{F}_{\e_j}(w_j^h)$ from above. 
Since $f(\xi, 0)=0\ \forall \xi\in\rr^d$, we  have that
 
  \[
\mathcal{ F}_{\e_j}(w_j^h)=\!\int_{\rr^d}\!\int_{\{x\in \Omega_{\e_j}(\xi) : |u_j(x)|\wedge |u_j(x+\e_j\xi)| \le M^{h+1}\}}\!f\left(\xi,\frac{\Phi_{M}^h(u_j(x+\e_j\xi))-\Phi_{M}^h(u_j(x))}{\e_j}\right)dx\,d\xi. 
 \]
Now, for $\xi\in\rr^d$, we distinguish in $\Omega_{\e_j}(\xi)$ the points where $|u_j(x)|\le |u_j(x+\e_j\xi)|$ from those where $|u_j(x)|> |u_j(x+\e_j\xi)|$ and we perform a similar analysis in both the two sets.  

To this end let us introduce the notation 
\[
\begin{array}{ll}
&\Omega^+_{\e_j}(\xi)=\{x\in \Omega_{\e_j}(\xi) : |u_j(x)|\le |u_j(x+\e_j\xi)| \},\\
 \\
& \Omega^-_{\e_j}(\xi)= \Omega_{\e_j}(\xi)\setminus   \Omega^+_{\e_j}(\xi)  =\{x\in \Omega_{\e_j}(\xi) : |u_j(x)|> |u_j(x+\e_j\xi)| \}.
\end{array}
\]
The sets $\Omega^+_{\e_j}(\xi)\cap \{|u_j(x)|\wedge |u_j(x+\e_j\xi)| \le M^{h+1}\}$ and $\Omega^-_{\e_j}(\xi)\cap \{|u_j(x)|\wedge |u_j(x+\e_j\xi)| \le M^{h+1}\}$  can be in turn decomposed,  respectively, as the union of the disjoint sets
\[
\begin{array}{ll}
&S_{1,h,j}^+(\xi)=\{x\in \Omega^+_{\e_j}(\xi)  : |u_j(x+\e_j\xi)|<M^h\}\\
&S_{2,h,j}^+(\xi)=\{x\in \Omega^+_{\e_j}(\xi) : |u_j(x)|<M^h, |u_j(x+\e_j\xi)|\ge M^{h+1}\}\\
&S_{3,h,j}^+(\xi)=\{x\in \Omega^+_{\e_j}(\xi) : |u_j(x)|<M^h\le |u_j(x+\e_j\xi)|< M^{h+1}\}\\
&S_{4,h,j}^+(\xi)=\{x\in \Omega^+_{\e_j}(\xi) : M^h\le |u_j(x)|\le |u_j(x+\e_j\xi)|\le M^{h+1}\}\\
&S_{5,h,j}^+(\xi)=\{x\in \Omega^+_{\e_j}(\xi) : M^h\le |u_j(x)|<M^{h+1}\le |u_j(x+\e_j\xi)|\}\\
\end{array}
\]
and 
\[
\begin{array}{ll}
&S_{1,h,j}^-(\xi)=\{x\in \Omega^-_{\e_j}(\xi)  : |u_j(x)|<M^h\}\\
&S_{2,h,j}^-(\xi)=\{x\in \Omega^-_{\e_j}(\xi) : |u_j(x+\e_j\xi)|<M^h, |u_j(x)|\ge M^{h+1}\}\\
&S_{3,h,j}^-(\xi)=\{x\in \Omega^-_{\e_j}(\xi) : |u_j(x+\e_j\xi)|<M^h\le |u_j(x)|< M^{h+1}\}\\
&S_{4,h,j}^-(\xi)=\{x\in \Omega^-_{\e_j}(\xi) : M^h\le |u_j(x+\e_j\xi)|\le |u_j(x)|\le M^{h+1}\}\\
&S_{5,h,j}^-(\xi)=\{x\in \Omega^-_{\e_j}(\xi) : M^h\le |u_j(x+\e_j\xi)|<M^{h+1}\le |u_j(x)|\}.\\
\end{array}
\]
Hence, using the growth assumption on $f$ and the Lipschitz continuity of $\Phi_{h,M}$, we have
\[
\begin{split}
&\mathcal{F}_{\e_j}(w_j^h)\le \int_{\rr^d}\int_{S_{1,h,j}^{\pm}(\xi)}f\left(\xi,\frac{u_j(x+\e_j\xi)-u_j(x)}{\e_j}\right)dx\,d\xi+ \int_{\rr^d}M(\xi)\int_{S_{2,h,j}^+(\xi)}\left|\frac{u_j(x)}{\e_j}\right|^pdx\,d\xi \\
&+\int_{\rr^d}M(\xi)\int_{S_{2,h,j}^-(\xi)}\left|\frac{u_j(x+\e_j\xi)}{\e_j}\right|^pdx\,d\xi+\sum_{i=3}^5\int_{\rr^d}M(\xi)\int_{S_{i,h,j}^{\pm}(\xi)}\left|\frac{u_j(x+\e_j\xi)-u_j(x)}{\e_j}\right|^pdx\,d\xi, 
\end{split}
\]
where for the sake of notation we have set $S_{i,h,j}^{\pm}(\xi)=S_{i,h,j}^+(\xi)\cup S_{i,h,j}^-(\xi)$. 

Let us now sum over $h=1,\cdots, M$ and  get 
\begin{equation}\label{sum}
\begin{split}
\frac{1}{M}\sum_{h=1}^M\mathcal{F}_{\e_j}(w_j^h)&\le \mathcal{F}_{\e_j}(u_j)+\frac{6}{M}\int_{\rr^d}M(\xi)\int_{\Omega_{\e_j}(\xi)}\left|\frac{u_j(x+\e_j\xi)-u_j(x)}{\e_j}\right|^pdx\,d\xi\\
&+\frac{1}{M}\sum_{h=1}^M\int_{\rr^d}M(\xi)\Big(\int_{S_{2,h,j}^{+}(\xi)}\left|\frac{u_j(x)}{\e_j}\right|^pdx + \int_{S_{2,h,j}^{-}(\xi)}\left|\frac{u_j(x+\e_j\xi)}{\e_j}\right|^pdx\Big)\,d\xi,\\
\end{split}
\end{equation}
since the families $\{S_{i,h,j}^+(\xi)\}_{h\in\NN}$ and $\{S_{i,h,j}^-(\xi)\}_{h\in\NN}$, with $i=3,4,5$, consist  of pairwise disjoint sets. 
Using Corollary \ref{boundlemma-lip}, \eqref{bound1} and \eqref{bound2}, the second term in the right handside of \eqref{sum} can be estimated from above by 
\begin{equation}\label{bound4}
\begin{split}
\frac{6}{M}\int_{\rr^d}M(\xi)& \int_{\Omega_{\e_j}(\xi)}\left|\frac{u_j(x+\e_j\xi)-u_j(x)}{\e_j}\right|^pdx\,d\xi\\
&\le \frac{6}{M}C(\Omega, r_0)\int_{\rr^d}M(\xi)(|\xi|^p+1)\,d\xi \left(G^{r_0,p}_{\e_j}(u_j,\Omega)+\|u_j\|^p_{L^p(\Omega,\rr^m)}\right)\\
& \le\frac{1}{M} C_2\,\left(G^{r_0,p}_{\e_j}(u_j,\Omega)+\|u_j\|^p_{L^p(\Omega,\rr^m)}\right)\le \frac{1}{M}C_1C_2,
 \end{split}
\end{equation}
if $\e_j<\e_0$, with $\e_0$ as in Corollary \ref{boundlemma-lip}. Since by \eqref{bound3} we have  that   $\displaystyle\frac{C_1\, C_2}{M} <\frac{\eta}{2}$,  we get for $\e_j\in (0,\e_0)$
\begin{equation}\label{chiellini}
\frac{6}{M}\int_{\rr^d}M(\xi) \int_{\Omega_{\e_j}(\xi)}\left|\frac{u_j(x+\e_j\xi)-u_j(x)}{\e_j}\right|^pdx\,d\xi< \frac{\eta}{2}.  
\end{equation}
We are left with the estimate of the last term in the right handside of  \eqref{sum}. Arguing as in  \eqref{bound4} and taking into account \eqref{bound2}, we deduce that 
for $\e_j\in (0,\e_0)$
\[
\begin{split}
C_1 C_2&\ge\int_{\rr^d}M(\xi)\int_{S_{2,h,j}^{\pm}(\xi)}\left|\frac{u_j(x+\e_j\xi)-u_j(x)}{\e_j}\right|^pdx\,d\xi\\
&\ge  \int_{\rr^d}M(\xi)\int_{S_{2,h,j}^{\pm}(\xi)}\left|\left|\frac{u_j(x+\e_j\xi)}{\e_j}\right|-\left|\frac{u_j(x)}{\e_j}\right|\right|^pdx\,d\xi\\
&\ge  \int_{\rr^d}M(\xi)\int_{S_{2,h,j}^{\pm}(\xi)}\left(\frac{M^{h+1}-M^h}{\e_j}\right)^pdx\,d\xi=\frac{(M^{h+1}-M^h)^p}{\e_j^{p}}\int_{\rr^d}M(\xi)|S_{2,h,j}^{\pm}(\xi)|\,d\xi,
\end{split}
\]
which in turn yields
$$
\int_{\rr^d}M(\xi)|S_{2,h,j}^{\pm}(\xi)|\,d\xi\le \frac{C_1 C_2\e_j^{p}}{(M^{h+1}-M^h)^p}.
$$
Exploiting this last estimate and the very definition of  $S_{2,h,j}^{+}(\xi), S_{2,h,j}^{-}(\xi)$,  we also get for $\e_j\in (0,\e_0)$
\begin{equation}
\begin{split}\label{buffon}
&\int_{\rr^d}M(\xi)\left( \int_{S_{2,h,j}^{+}(\xi)}\left|\frac{u_j(x)}{\e_j}\right|^pdx+ \int_{S_{2,h,j}^{-}(\xi)}\left|\frac{u_j(x+\e_j\xi)}{\e_j}\right|^pdx\right)\,d\xi\\ 
&\le \frac{M^{hp}}{\e_j^{p}}\int_{\rr^d}M(\xi)|S_{2,h,j}^{\pm}(\xi)|\,d\xi\le C_1\, C_2\frac{M^{hp}}{(M^{h+1}-M^h)^p}=\frac{C_1\, C_2}{(M-1)^p}.\\
\end{split}
\end{equation}
Since, by \eqref{bound3}, $M\ge 2$, we have 
\begin{equation}\label{bonucci}
\frac{C_1\, C_2}{(M-1)^p}\le \frac{C_1\, C_2}{(M-1)}\le  \frac{C_1\, C_2}{\lfloor\frac{ 2\, C_1C_2}{\eta}\rfloor+1}< \frac{\eta}{2}.
\end{equation}
Hence, by \eqref{sum}, \eqref{chiellini}, \eqref{buffon} and \eqref{bonucci}, we eventually deduce that for every $j\in\NN$ such that $\e_j<\e_0$ there exists $h(j)\in\{1,\cdots, M\}$ satisfying
\[
\mathcal{F}_{\e_j}(w_j^{h(j)})\le\frac{1}{M}\sum_{h=1}^M\mathcal{F}_{\e_j}(w_j^h)\le \mathcal{F}_{\e_j}(u_j)+\eta.
\]
We then define $\Phi_{j,M}=\Phi_M^{h(j)}$. 
Up to selecting a subsequence, we may also assume that $h(j)$ is a constant value in $\{1,\cdots, M\}$.  
\end{proof}

\subsection{Approximating capacitary energy densities}
In this subsection we introduce and investigate the main properties of suitable energy densities defined through minimum problems of capacitary type involving the approximating energies $\mathcal{F}_\e^T$.  

 
\medskip 

\noindent For any $\e>0$, $T>r_0$, $R\ge 2+T\e$, and $z\in\rr^m$, set
\begin{equation}\label{approxdens}
\varphi_{\e,T, R}(z):=\inf\{\mathcal{F}_\e^T(v, B_R):\ v\in L^p_{\e,T, z}(B_R;\rr^m)\},
\end{equation}
where $\mathcal{F}_\e^T$ is defined in \eqref{uncofunctionalsT} and 
\begin{equation}\label{testfunct}
L^p_{\e,T, z}(B_R;\rr^m):=\{v\in L^p(B_R;\rr^m):\ v\equiv 0\ \text{in}\ B_1,\ v\equiv z\ \text{on}\ \partial^{\e T}B_R\}, 
\end{equation}
with the notation $ \partial^{\e T} A:=\partial A+B_{\e T}$,  for any  $A\subset \mathbb{R}^d $.  

We identify any $v\in L^p_{\e,T, z}(B_R;\rr^m)$ with its extension  in $L^{p}_{\loc}(\mathbb{R}^d;\mathbb{R}^m)$ such that 
$ v \equiv z$ in $\rr^d\setminus B_R$. Hence the function  $\varphi_{\e,T, R}(z)$ can be also rewritten as 
\begin{equation}\label{equivvarphi}
\varphi_{\e,T, R}(z)=\inf\{\mathcal{F}_\e^T(v,\rr^d) :   v-z\in L^{p}(\mathbb{R}^d;\mathbb{R}^m),  v\equiv 0\ \text{in}\ B_1, v \equiv z \ \text{in}\ \rr^d\setminus B_{R-\e T} \}.
\end{equation}
Note that the request $R\ge 2+T\e$  is not restrictive, as we are interested in letting $R\to +\infty$; this assumption will be useful in Proposition \ref{pgrowthdens}.

\begin{remark}\label{inf=min}
Note that,  if $f(\xi,\cdot )$ is convex for every  $\xi\in \rr^d$, the infimum defining $\varphi_{\e,T, R}(z)$ in \eqref{approxdens}  is actually a minimum. Indeed, by the convexity of $f(\xi, \cdot)$ also  $\mathcal{F}_\e^T(v, B_R)$  is convex. Hence, taking into account  Proposition \ref{pwineq} with $E=B_1$, $A=B_R$, $\mathcal{F}_\e^T(v, B_R)$ is lower semicontinuous and coercive with respect to the weak topology in $ L^p(B_R;\rr^m)$. As the constraints $v\equiv 0\ \text{in}\ B_1,\ v\equiv z\ \text{on}\ \partial^{\e T}B_R$ are convex  and closed by the strong  convergence in $ L^p(B_R;\rr^m)$ the existence of minimizers follows by the standard methods. 
\end{remark}
The properties  of the densities  $\varphi_{\e,T, R}$ we are going to state and prove  will be instrumental in Subsection \ref{limitcap}  in studying the pointwise  and locally uniform limit of $\varphi_{\e,T, R}(\cdot)$ when the parameters $R,T$ go to $ +\infty$, and  $\e$ either goes to $0$ or remains fixed equal to $\alpha$. 
These results  will allow us to estimate the energetic contribution near the perforations leading to the appearance  of  the density functions  $\varphi$ and $\varphi_{NL,\alpha}$ defined by \eqref{denscapacitaria} and \eqref{denscapacitariaNL}, respectively. 

\medskip

\noindent The first result  establishes growth conditions of order $p$ of $\varphi_{\e,T, R}$. 
\begin{proposition}\label{pgrowthdens}
Let $f$ satisfy assumptions {\rm ({\bf H})} and {\rm({\bf G})}, and let $T>r_0, \e_0>0$, and $R>1$ be fixed such that $R-\e_0 T \geq 2$. Then, for every $0<\e\le \e_0$ there exists $c_1, c_2>0$ such that 
\begin{equation}\label{growthphibelow}
c_1|z|^p\leq \varphi_{\e,T, R}(z)\ \forall z\in\rr^m 
\end{equation}
\begin{equation}\label{growthphiabove} 
\varphi_{\e,T, R}(z)\leq c_2|z|^p\ \forall z\in\rr^m.
\end{equation}
In particular, the constant $c_1$ depends on $p,d,\lambda_0,r_0,\e_0$ and the constant $c_2$ on $p,d,r_0$.
\end{proposition}
\begin{proof}
We first prove \eqref{growthphibelow}. The proof relies on a suitable lower bound of $\mathcal{F}_\e(v,B_R)$ with discrete energies. In order to avoid too many technicalities, we restrict the proof to the case $d=2$; the argument can be generalised to any dimension (see e.g. the proof of Theorem 2.6 in \cite{solci}). Let us introduce some notation. Given $\xi\in\rr^2\setminus \{0\}$, let ${\cal L}_\xi$ be the lattice in $\rr^2$ defined by
$$
{\cal L}_\xi={\mathbb Z}\xi\oplus{\mathbb Z}\xi^\perp,
$$
where $\xi^\perp:=(-\xi_2,\xi_1)$.

Let $v\in L^p_{\e, T, z}(B_R;\rr^m)$ and let $0<\bar r<r_0$. By (\textbf{H}) and (G0), we get
\begin{equation}\label{stimadisc}
\begin{split}
\mathcal {F}^T_\e(v,B_R)&\geq \lambda_0\, G_\e^{\bar r,p}(v,B_R)= \lambda_0\, G_\e^{\bar r,p}(v,\rr^2)\\
&=\frac{\lambda_0}{2}\int_{B_{\bar r}}\int_{\rr^2}\sum_{\xi'\in\{\xi,\xi^\perp\}} \left |\frac{v(x+\e\xi')-v(x)}{\e}\right|^p\, dx\, d\xi \\
&= \frac{\lambda_0}{2}\int_{B_{\bar r}}\sum_{k\in {\cal L}_\xi }\int_{\e(k+ Q^\xi)}\sum_{\xi'\in\{\xi,\xi^\perp\}} \left |\frac{v(x+\e\xi')-v(x)}{\e}\right|^p\, dx\, d\xi 
\end{split}
\end{equation}
where
$$
Q^\xi:=[0,1)\xi\oplus [0,1)\xi^\perp.
$$
Let $T_\e^\xi v$ the function which is constant on each square $\e(k+ Q^\xi),\ k\in {\cal L}_\xi$, and is defined by
$$
T_\e^\xi v (x)=\frac{1}{(\e|\xi|)^2}\int_{\e(k+ Q^\xi)} v(y)\, dy,\quad x\in \e(k+ Q^\xi),\ k\in {\cal L}_\xi.
$$
Then, by \eqref{stimadisc} and Jensen's inequality, we get
\begin{equation}\label{stimadisc2}
\begin{split}
\mathcal {F}_\e^T(v,B_R) &\geq\frac{\lambda_0}{2}\int_{B_{\bar r}}\sum_{k\in {\cal L}_\xi}(\e|\xi|)^2\sum_{\xi'\in\{\xi,\xi^\perp\}} \left |\frac{T_\e^\xi v(\e(k+\xi'))-T_\e^\xi v(\e k)}{\e}\right|^p\, d\xi \\
&=\frac{\lambda_0}{2}\int_{B_{\bar r}}|\xi|^p\sum_{k\in {\cal L}_\xi}(\e|\xi|)^2\sum_{\xi'\in\{\xi,\xi^\perp\}} \left |\frac{T_\e^\xi v(\e(k+\xi'))-T_\e^\xi v(\e k)}{\e|\xi|}\right|^p\, d\xi. 
\end{split}
\end{equation}
Let ${\cal T}_\xi^{\pm}$ be the triangles defined by
$$
{\cal T}_\xi^{\pm}:=\{x\in Q^\xi:\ \pm\langle x,\xi\rangle\leq \pm \langle x,\xi^\perp\rangle\ \},
$$
and let $w_\e^\xi$ the piecewise affine function obtained by linearly interpolating the values $\{T_\e^\xi v(\e k)\}_{k\in {\cal L}_\xi}$ on the triangles $\e (k+{\cal T}_\xi^{\pm})$, $k\in{\cal L}_\xi$. Note that, for $\e\le\e_0$, if $\bar r\le\frac{1}{2\e_0}$ and $\xi\in B_{\bar r}$, then $w_\e^\xi \equiv 0$ on $B_{\frac12}$ and  $w_\e^\xi \equiv z$ on $\rr^2\setminus B_{\frac32 R}$. Moreover, taking into account \eqref{stimaciuno} and \eqref{capmonot},  we easily infer that
\begin{equation}\label{stimadisc3}
\begin{split}
\sum_{k\in {\cal L}_\xi}(\e|\xi|)^2 & \sum_{\xi'\in\{\xi,\xi^\perp\}} \left |\frac{T_\e^\xi v(\e(k+\xi'))-T_\e^\xi v(\e k)}{\e|\xi|}\right|^p
\geq C \int_{\rr^2} |\nabla w_\e^\xi|^p\, dx\\
&\geq C\,  \text{cap}_p (B_{\frac12},B_{\frac32 R})|z|^p\geq  C\,  \text{cap}_p (B_{\frac12})|z|^p.
\end{split}
\end{equation}
In conclusion, selecting $\bar r:=\max\{r_0, \frac{1}{2\e_0}\}$,  we get 
$$
\mathcal {F}^T_\e(v,B_R)\ge \frac{\lambda_0}{2}\left(\int_{B_{\bar r}}|\xi|^p \, d\xi \right) C\,\text{cap}_p (B_{\frac12})|z|^p =  c_1 |z|^p.
$$

\noindent We now prove \eqref{growthphiabove}. To this aim, note that, by using a Fubini argument, one can easily shows that there exists $C=C(r_0)$ such that, for $\e\le\e_0$, for any $u$ such that  $u-z\in C^1_c(B_{R-\e T};\rr^m)$, then
$$
G_\e^{r_0,p}(u,\rr^d)=G_\e^{r_0,p}(u,B_{R-\e T})\leq C \int_{B_{R-\e T}}|\nabla u|^p\, dx.
$$ 
By (\textbf{H}), (G1),  \eqref{eqtuttoerred}, and  the density of functions compactly supported in the capacitary problem, we then get
\[
\begin{split}
& \varphi_{\e,T,R}(z)  \leq {\mathcal F}^T_\e (u, B_R)\le C \left(\int_{\rr^d} M(\xi)(|\xi|^p+1)\, d\xi\right)  G_\e^{r_0,p}(u,\rr^d) \\
&\leq C \left(\int_{\rr^d} M(\xi)(|\xi|^p+1)\, d\xi\right)\inf\Big\{\int_{B_{R-\e T}}|\nabla u|^p\, dx:\, u\equiv 0\ \text{in } B_1,\, u-z\in C^1_c(B_{R-\e T};\rr^m)\Big\}\\
&= C \left(\int_{\rr^d} M(\xi)(|\xi|^p+1)\, d\xi\right) \text{cap}_{p}(B_1,B_{ R-\e T})|z|^p\\
& \leq C \left(\int_{\rr^d} M(\xi)(|\xi|^p+1)\, d\xi\right) \text{cap}_{p}(B_1,B_2)|z|^p=c_2|z|^p.
\end{split}
\]
\end{proof}

\noindent In the next proposition we show that the functions $\varphi_{\e,T,R}$ are uniformly Lipschitz continuous on compact sets. 
\begin{proposition}\label{th:equilip}
Let $f$ satisfy assumptions {\rm ({\bf H})}, {\rm ({\bf G})} and {\rm ({\bf L})}. Then there exist a constant $C>0$ independent of $\e,T$ and $R$ such that for every $z,w\in\rr^m$ we have
\begin{equation}\label{equilip}
|\varphi_{\e,T, R}(w)-\varphi_{\e, T, R}(z)|\leq C(|z|^{p-1}+|w|^{p-1}|) |w-z|.
\end{equation}
\end{proposition}
\begin{proof}

Let us prove \eqref{equilip} for fixed $z$ and $w$. Since the inequality is trivially true when $z=0$ or $w=0$, we may suppose both not null and consider the map $\phi:\rr^m\to\rr^m$ defined by
$$
\phi(\zeta)=\frac{|w|}{|z|}\mathcal{R}_z^w(\zeta),
$$
where $\mathcal{R}_z^w$ is a rotation that maps $z$ into $\frac{|z|}{|w|} w$. Note that $\phi(0)=0$, $\phi(z)=w$ and 
\begin{equation}\label{similitude}
\|\nabla \phi\|_\infty\leq C \frac {|w|}{|z|},\quad \|\nabla \phi-I\|_\infty\leq C \frac {|w-z|}{|z|} .
\end{equation}
For $\eta>0$, let $v_z\in L^p_{\e,T, z}(B_R;\rr^m)$ be such that $\mathcal{F}_\e^T(v_z, B_R)\le \varphi_{\e,T, R}(z)+\eta$ and set $$
v_w:=\phi\circ v_z.
$$
Note that $v_w\in L^p_{\e,T,w}(B_R;\rr^m)$, hence
\begin{equation}\label{stimaphi}
\varphi_{\e,T, R}(w)\leq \mathcal{F}_\e^T(v_w, B_R)\leq \varphi_{\e,T, R}(z)+\mathcal{F}_\e^T(v_w, B_R)-\mathcal{F}_\e^T(v_z, B_R)+\eta.
\end{equation}
By hypothesis  ({\bf L}) and \eqref{similitude}, we infer that for every $\xi\in\rr^d$
\[
\begin{split} 
&|f(\xi,D_\e^\xi v_w(x))-f(\xi,D_\e^\xi v_z(x))|  \leq C M(\xi)(|D_\e^\xi v_z(x)|^{p-1}+|D_\e^\xi v_w(x)|^{p-1}) |D_\e^\xi v_z(x)-D_\e^\xi v_w(x)|\\
& \qquad  \leq C M(\xi)(|D_\e^\xi v_z(x)|^{p-1}+\|\nabla \phi\|_{\infty}^{p-1}|D_\e^\xi v_z(x)|^{p-1}) \|\nabla \phi-I\|_{\infty}|D_\e^\xi v_z(x)|\\
&\qquad \leq C M(\xi)\frac{|z|^{p-1}+|w|^{p-1}}{|z|^p}|w-z||D_\e^\xi v_z(x)|^p.\\
\end{split}
\]
Thus, by \eqref{stimaphi}, we get
\begin{equation}\label{stimaphi2}
\varphi_{\e,T, R}(w)\leq  \varphi_{\e,T, R}(z)+C \frac{|z|^{p-1}+|w|^{p-1}}{|z|^p}|w-z|\int_{\rr^d} M(\xi)\int_{ (B_R)_{\e}(\xi)}|D_\e^\xi v_z(x)|^p\, dx\, d\xi+\eta.
\end{equation}
By (\textbf{H}), (G0) and \eqref{growthphiabove}, we get
\begin{equation}\label{GF}
G_\e^{r_0,p}(v_z,B_R)\leq C \, \mathcal{F}_\e^T(v_z,B_R)\le C\, (\varphi_{\e,T, R}(z)+\eta)\leq C(|z|^p+\eta).
\end{equation}
Since, by  \eqref{eqtuttoerred}, we have that for any $\xi\in\rr^d$
$$
\int_{ (B_R)_{\e}(\xi)}|D_\e^\xi v_z(x)|^p\, dx\leq C (|\xi|^p+1)G_\e^{r_0,p}(v_z,B_R),
$$
inequality \eqref{stimaphi2} and (G1) yields that
\begin{equation*}
\varphi_{\e,T, R}(w)\leq  \varphi_{\e, T, R}(z)+C (|z|^{p-1}+|w|^{p-1})|w-z|\frac{|z|^p+\eta}{|z|^p}+\eta.
\end{equation*}
Taking the limit as  $\eta$ tends to 0, and then reversing the role of $z$ and $w$, \eqref{equilip} easily follows.
\end{proof}

We conclude this subsection with a technical result, yielding that in the minimum problems defining $\varphi_{\e,T,R}$ we may reduce to admissible functions uniformly bounded in $L^\infty$. The strategy of the proof is analogous to that of Lemma \ref{truncation}, hence we highlight only the main differences.

\begin{proposition}\label{potcaplimitato}
Let $T>r_0$, $\alpha>0$, $ R\ge 2+T\alpha$ and $\bar C, M_0>0$ be fixed. Then for every $\eta>0$ there exists $M>M_0$ such that for every  $z\in B_{M_0}$, given $v\in L^p_{\alpha,T,z}(B_R;\rr^m)$  such that 
$$
\mathcal{F}_\alpha^T( v,B_R)\le \bar C |z|^p,
$$ 
then there exists $v_M\in L^p_{\alpha,T,z}(B_R;\rr^m)$, with $\|v_M\|_{L^\infty(B_R;\rr^m)}\le M$, such that
$$
\mathcal{F}^T_\alpha(v_M,B_R)\le  \mathcal{F}_\alpha^T( v,B_R)+\eta.
$$
\end{proposition}
\begin{proof}
Given $z\in B_{M_0}$ and $R>0$, by (\textbf{H}) and (G0) we have  that 
$$
G^{r_0,p}_{\alpha}( v,B_R)\le C|z|^p,
$$
where the constant $C$ depends only on $\bar C, r_0,\lambda_0$. So now it suffices to retrace the steps of the proof of Lemma \ref{truncation}, replacing  the constants in \eqref{bound1} and \eqref{bound2} with 
$$
C_1:=\sup_R G^{r_0,p}_{\alpha}( v,B_R)<+\infty,
$$
and 
$$
C_2= 6C(r_0) \int_{\rr^d}M(\xi)(|\xi|^p+1)\,d\xi , 
$$
respectively, where $C(r_0)$ is the constant obtained in Remark \ref{tuttolospazio}, and using Remark \ref{tuttolospazio} instead of Corollary \ref{boundlemma-lip}. The function $v_M$ is obtained through $\Phi_M( v)$ with a suitable choice of $M$.
\end{proof}

\subsection{Asymptotics of the approximating capacitary energy density}\label{limitcap}
%

We now show that, if $R_\e\to +\infty$ as $\e\to 0$ the functions $\varphi_{\e,T, R_\e}(z)$ approximate the energy density $\varphi^T(z)$ defined  in \eqref{denscapacitariaT}. A crucial role in the proof is played by Corollary \ref{compactnessloc}. 


\begin{proposition}\label{capterm}
Let $\varphi^T$ and $\varphi_{\e,T, R}$ be defined by \eqref{denscapacitariaT} and  \eqref{approxdens}, respectively. Then, if $R_\e\to +\infty$ as $\e\to 0$, it holds
\begin{equation}\label{convunifcap}
\lim_{\e\to 0} \varphi_{\e,T, R_\e} (z)=\varphi^T (z)
\end{equation}
uniformly on compact sets.
\end{proposition}
\begin{proof}
By Proposition \ref{th:equilip} it suffices to prove that \eqref{convunifcap} holds pointwise. We will show that it is a consequence of Theorem \ref{convexhom},  and Corollary \ref{compactnessloc}. With fixed $z\in\rr^m$, let $v_\e\in L^p_{\e,T, z} (B_{R_\e};\rr^m)$ be such that
$$
\varphi_{\e, T, R_\e} (z)=\mathcal{F}_\e^T(v_\e,B_{R_\e}) +o(\e),
$$
and let $u_\e\in L^p(\rr^d;\rr^m)$ equal to $v_\e-z$ on $B_{R_\e}$ and $u_\e\equiv 0$ on $\rr^d\setminus B_{R_\e}$. By (\textbf{H}), (G0) and \eqref{growthphiabove}, it holds 
\begin{equation}\label{boundueps}
\sup_\e G_\e^{r_0,p}(u_\e, \rr^d) <+\infty.
\end{equation}
By Theorem \ref{pwineq} applied with $A$ any bounded  open Lipschitz set   in $\rr^d$ and  $E=B_1$, we get 
$$
\sup_\e \|u_\e\|_{L^p(A;\rr^m)} <+\infty.
$$
Thus, using Corollary \ref{compactnessloc} we get that, up to a subsequence (not relabelled), $u_\e$ converge in $L^p_{\loc} (\rr^d;\rr^m)$  to a function $u\in W^{1,p}_{\loc}(\rr^d;\rr^m)\cap L^{p^*}(\rr^d;\rr^m)$ such that $u= -z$ on $B_1$. Moreover, with fixed $R>0$,  by Theorem \ref{convexhom}, we deduce that 
$$
\liminf_{\e\to 0} \varphi_{\e, T,R_\e} (z)=\liminf_{\e\to 0}\mathcal{F}^T_\e(v_\e,B_{R_\e})\ge\liminf_{\e\to 0}\mathcal{F}^T_\e(v_\e,B_{R})\ge\int_{B_R} f^T_{hom}(\nabla u)\, dx. 
$$
Letting $R\to +\infty$ we obtain 
$$
\liminf_{\e\to 0} \varphi_{\e, T,R_\e} (z)\ge\int_{\rr^d} f^T_{hom}(\nabla u)\, dx \ge \varphi^T(z).
$$
We now claim that 
$$
\varphi^T(z)=
\inf\Big\{\int_{\mathbb{R}^d} f_{hom}^T(\nabla u)\, dx: u\equiv -z\ \text{in}\ B_1,\   u\in W^{1,p}(\mathbb{R}^d;\mathbb{R}^m)  \hbox{ compactly supported}\Big\}. 
$$

To this aim, let us us fix  a cut-off function $\zeta$  between $B_1$ and $B_2$ and   $u\in L^{p*}(\mathbb{R}^d;\mathbb{R}^m)\cap W^{1,p}_{\loc}(\mathbb{R}^d;\mathbb{R}^m)$, $u\equiv -z$ in $B_1$,  with $\int_{\mathbb{R}^d} f_{hom}^T(\nabla u)\, dx <+\infty$. Note that, taking into account Remark \ref{phomfhom}, $\nabla u\in L^p(\rr^d; \rr^{d\times m})$.  We now set, for any $n\in \NN$, 
$u_n(x)=\zeta(x/n) u(x)$. An easy computation shows that  $u_n\in W^{1,p}(\mathbb{R}^d;\mathbb{R}^m)$, $u_n$ is compactly supported and $u_n\equiv -z$ in $B_1$. Moreover it holds  that $\nabla u_n\to \nabla u$ strongly in $ L^p(\rr^d; \rr^{d\times m})$. Indeed, by H\"older inequality, we have  
\[
\begin{split}
\int_{\rr^d}|\nabla u_n-\nabla u|^p \, dx &\le C\int_{\rr^d\setminus B_n}   |\nabla u|^p\, dx + \frac{C}{n^p}\int_{B_{2n}\setminus B_n}| u|^p\, dx \\
&\le C\int_{\rr^d\setminus B_n}   |\nabla u|^p\, dx + \frac{C}{n^p}|B_{2n}\setminus B_n|^{1-\frac{p}{p^*}}\Big(\int_{B_{2n}\setminus B_n}| u|^{p^*}\, dx \Big)^{\frac{p}{p^*}}\\
& \le C\int_{\rr^d\setminus B_n}   |\nabla u|^p\, dx + C\Big(\int_{B_{2n}\setminus B_n}| u|^{p^*}\, dx \Big)^{\frac{p}{p^*}},  \end{split}
\]
and the last two terms tend to $0$ as $n\to +\infty$.  The claim follows by using the dominated convergence theorem together with Remark \ref{phomfhom}.

Thus, taking the claim into account and  using a density argument, given $\eta>0$, we may assume that  there exists $u\in C^\infty_c(\rr^d;\rr^m)$ such that $u\equiv - z$ on $B_1$ and
$$
\int_{\rr^d} f^T_{hom}(\nabla u)\, dx\le \varphi^T(z)+\eta.
$$
Let $\bar R>0$ such that $\supp\, u\subseteq B_{\bar R}$. 
Then, by  \cite[Proposition 5.3]{AABPT} applied with $A=B_{\bar R}\setminus B_1$,  there exists a family of functions $u_\e\in L^p(B_{\bar R};\rr^m)$, with $u_\e\equiv -z$ on $B_1$ and $u_\e\equiv 0$ on $\partial^{\e T} B_{\bar R}$ such that
$$
\lim_{\e\to 0}\mathcal{F}^T_\e(u_\e,\rr^d)=\int_{\rr^d} f^T_{hom}(\nabla u)\, dx.
$$
Hence, set $v_\e:=u_\e+z$, we get that $v_\e\in L^p_{\e,T, z} (B_{R_\e};\rr^m)$ for $\e$ small enough, thus
$$
\limsup_{\e\to 0} \varphi_{\e, T, R_\e} (z)\leq\limsup_{\e\to 0}\mathcal{F}^T_\e(v_\e, B_{R_\e})\le\int_{\rr^d} f^T_{hom}(\nabla u)\, dx\le \varphi^T(z)+\eta
$$
and the thesis follows by the arbitrariness of $\eta>0$.
\end{proof}

\noindent Note that for any $\alpha>0$, $T>r_0$ and $z\in\rr^m$, the function $R\in [2+\alpha T,+\infty)\mapsto \varphi_{\alpha,T,R}(z)$ is decreasing, as it is easily seen  by \eqref{equivvarphi}. Hence, for $z\in \rr^m$, it is well defined $\displaystyle\lim_{R\to +\infty} \varphi_{\alpha,T,R}(z)$ and this convergence is also locally uniform by Proposition \ref{th:equilip}. One can easily shows that
$$
\varphi_{NL,\alpha}^T(z)= \lim_{R\to +\infty} \varphi_{\alpha,T,R}(z),  
$$
where $\varphi_{NL,\alpha}^T$ is defined by \eqref{denscapacitariaNLT}.
\medskip



 
\noindent The properties of the densities $\varphi_{\e,T,R}$ obtained so far allow to prove the following result about the $L^1$-convergence of suitable Riemann sums to the capacitary densities $\varphi^T$ and $\varphi_{NL,\alpha}^T$.
\begin{proposition}\label{th:riemannsum}
Let $\e_j\to 0$ and  $R_j\to +\infty$ as $j\to +\infty$ and let $(u_j)$ be a bounded sequence in $L^\infty(\Omega;\rr^m)$ such that $\sup_j\mathcal{F}^T_{\e_j}(u_j)<+\infty$ and $u_j\to u$ in $L^p(\Omega;\rr^m)$, for some $u\in W^{1,p}(\Omega;\rr^m)$. Let $A_j^{i,k_i}$  and $u_j^{i,k_i}$, $i\in Z_j (\Omega)$,  be as in \eqref{anelli} and \eqref{medianelli}, respectively,  with $\rho_j=O(\delta_j)$ and  $h=k_i$, for an arbitrary choice of $k_i$. 
\begin{itemize}
\item[(i)] Let $\Psi^T_j:\Omega\to\rr$  be defined by
$$
\Psi^T_j(x)= \sum_{i\in Z_j (\Omega)} \varphi_{\e_j, T, R_j}(u_j^{i,k_i})\chi_{Q_{\delta_{\e_j}}(i)}(x),
$$
where
$$
Q_{\delta_{\e_j}}(i):=\delta_{\e_j}i +\delta_{\e_j} Q_1.
$$
Then $\Psi^T_j\to \varphi^T(u)$  in $L^1(\Omega)$.
\item[(ii)]Let $\Psi^T_{j,\alpha}:\Omega\to\rr$ be defined by
$$
\Psi^T_{j,\alpha}(x)= \sum_{i\in Z_j(\Omega)} \varphi_{\alpha, T, R_j}(u_j^{i,k_i})\chi_{Q_{\delta_{\e_j}}(i)}(x),
$$
Then  $\Psi^T_{j,\alpha}\to \varphi_{{NL},\alpha}^T(u)$ in $L^1(\Omega)$.
\end{itemize}
\end{proposition}
\begin{proof}
We prove (i), the proof of (ii) being analogous.  We have the following estimate
\[
\begin{split}
\int_\Omega | \Psi^T_j(x)-\varphi^T(u(x))|\, dx&\leq\sum_{i\in Z_j(\Omega)}\int_{Q_{\delta_{\e_j}(i)}} \left|\varphi_{\e_j, T,R_j}(u_j^{i,k_i})-\varphi_{\e_j, T, R_j}(u_j (x))\right|\, dx\\
&+\sum_{i\in Z_j(\Omega)}\int_{Q_{\delta_{\e_j}(i)}} \left|\varphi^T(u(x))-\varphi_{\e_j,T, R_j}(u_j(x))\right|\, dx\\
&+ \int_{\Omega\setminus\bigcup_{i\in Z_j(\Omega)}Q_{\delta_{\e_j}(i)}}|\varphi^T(u(x))|\, dx=:I^1_j+I^2_j+I^3_j.
\end{split}
\]
By Proposition \ref{capterm}, we easily deduce that $I^2_j\to 0$. Since $|\Omega\setminus \bigcup_{i\in Z_j(\Omega)}Q_{\delta_{\e_j}(i)}|\to 0$, we also infer that $I^3_j\to 0$. Finally, by Proposition \ref{th:equilip}, we may estimate $I^1_j$ as follows
\[
\begin{split}
I^1_j\leq C \sum_{i\in Z_j(\Omega)}\int_{Q_{\delta_{\e_j}(i)}} \left| u_j^{i,k_i}-u_j(x)\right|\, dx.
\end{split}
\]
By H\"older's inequality and Proposition \ref{pwscaling}, we have
\[
\begin{split}
\int_{Q_{\delta_{\e_j}(i)}} \left| u_j^{i,k_i}-u_j(x)\right|\, dx&\leq \delta_{\e_j}^{\frac{d(p-1)}{p}}\left(\int_{Q_{\delta_{\e_j}(i)}} \left| u_j^{i,k_i}-u_j(x)\right|^p\, dx\right)^{\frac1p}\\
&\leq C \delta_{\e_j}^{\frac{d(p-1)}{p}}\delta_{\e_j} \left(G_{\e_j}^{r_0,p} (u_j, Q_{\delta_{\e_j}(i)})\right)^{\frac1p},
\end{split}
\]
Hence
$$
I^1_j\leq C \delta_{\e_j} \left( G_{\e_j}^{r_0,p}(u_j, \Omega)\right)^{\frac1p}\le C \delta_{\e_j} \left( {\mathcal F}_{\e_j}^T(u_j)\right)^{\frac1p}\to 0.
$$

\end{proof}
\medskip
\noindent We conclude this subsection showing the convergence of  $\varphi_{NL,\alpha}^T$ to $\varphi_{NL,\alpha}$ as $T\to +\infty$.
 
\begin{proposition}\label{captermNL}
For any $z\in\rr^m$, it holds
\[
\lim_{T\to +\infty}\varphi_{NL,\alpha}^T(z)=\sup_{T>r_0}\varphi_{NL,\alpha}^T(z)=\varphi_{NL,\alpha}(z),
\]
where $\varphi_{NL,\alpha}(z)$ is defined in \eqref{denscapacitariaNL}.
\end{proposition}
\begin{proof}  We first observe that the function $\varphi_{NL,\alpha}^T$ is increasing in $T$, hence it is well defined the limit
$$
\displaystyle \lim_{T\to +\infty}\varphi_{NL,\alpha}^T(z)=\sup_{T>r_0}\varphi_{NL,\alpha}^T(z),
$$
for any $z\in\rr^m$. Since $\mathcal{F}_\alpha^T(v,\rr^d)\le \mathcal{F}_\alpha(v,\rr^d)$ for every $T>0$, 
$$
\sup_{T>r_0}\varphi _{NL,\alpha}^T(z)\le \varphi_{NL,\alpha}(z).
$$
Let $v$ be an admissible function in the minimum problem defining $\varphi_{NL,\alpha}(z)$. In particular, $v-z$ satisfies \eqref{eqtuttoerred}, that is
\[
\int_{\rr^d}\Big|\frac{v(x+\e\xi)-v(x)}{\e}\Big|^p dx \le C(|\xi|^p+1) G_\e^{r_0,p}(v,\rr^d),
\]
where the constant $C$ depends on $r_0$. We now multiply each side of the previous inequality by the growth function $M(\xi)$, we apply Remark \ref{remcrescita}, integrate on $\rr^d\setminus B_T$, and finally obtain 
\[
\int_{\{|\xi|\ge T\}}\int_{\rr^d}f\left (\xi,\frac{v(x+\e\xi)-v(x)}{\e}\right) dx\,d\xi \le C\int_{\{|\xi|\ge T\}}M(\xi)(|\xi|^p+1)d\xi \,G_\e^{r_0,p}(v,\rr^d).
\]
Thanks to (G1), this gives
$$
\mathcal{F}_\alpha(v,\rr^d)\le \mathcal{F}_\alpha^T(v,\rr^d)+ o(T) \,G_\e^{r_0,p}(v,\rr^d).
$$
We now choose a function $v_T$ such that $v_T\equiv 0$ in $B_1$, $v_T-z\in L^p(\rr^d;\rr^m)$, $v_T-z$ is compactly supported, and 
$$
\mathcal{F}^T_\alpha(v_T,\rr^d)\le \varphi_{NL,\alpha}^T(z)+o(T).
$$
Thus, in particular
\begin{equation}\label{FFT}
\begin{split}
\varphi_{NL,\alpha}(z)\le \mathcal{F}_\alpha(v_T,\rr^d)&\le  \mathcal{F}_\alpha^T(v_T,\rr^d)+ o(T) \,G_\e^{r_0,p}(v_T,\rr^d)\\
&\le  \varphi_{NL,\alpha}^T(z)+o(T)+ o(T) \,G_\e^{r_0,p}(v_T,\rr^d).
\end{split}
\end{equation}
By finally using that 
\[
G_\e^{r_0,p}(v_T,\rr^d)\le C\,\mathcal{F}_\alpha^T(v_T,\rr^d)\le C\,\varphi_{NL,\alpha}^T(z)+o(T)\le C(|z|^p+1),
\]
the desired conclusion follows letting $T$ tend to $+\infty$ in \eqref{FFT}.
\end{proof}

\section{Proof of Theorem \ref{th:main} and Theorem \ref{th:mainNL}}

\noindent We will prove the two statements simultaneously, distinguishing the two regimes provided by assumption \eqref{zero} and \eqref{alfa}, respectively, only when necessary. 

\noindent Dividing the proof into three steps, we  first show that it suffices to prove the theorems for the truncated functionals $F_{\e,\delta_\e}^T $, and then we deal separately with the $\Gamma$-$\liminf$ and the $\Gamma$-$\limsup$ inequalities.

\medskip

\noindent {\bf   Step 1.}} It is not restrictive to prove both theorems under the additional assumption that there exists $T>0$ such that $f(\xi,z)=0$ if $|\xi|>T$. 

\noindent Indeed, under the hypotheses of Theorem \ref{th:main}, assume that for every $T>0$
\begin{equation}\label{gammaconvexT}
\Gamma(L^p)-\lim_{\e\to 0} F_{\e,\delta_\e}^T (u)=
\begin{cases}\displaystyle \int_\Omega f_{hom}^T(\nabla u)\, dx +\int_\Omega \varphi^T(u)\, dx & \text{if}\ u\in W^{1,p}(\Omega;\mathbb{R}^m),\cr\
+\infty & \text{otherwise},
\end{cases}
\end{equation}
where $F_{\e,\delta}^T$ is defined by \eqref{functionals2T} and $f_{hom}^T$ and $\varphi^T$ are defined by  \eqref{homformT} and \eqref{denscapacitariaT}, respectively. By \cite[Lemma 5.1]{AABPT},
$$
\Gamma(L^p)-\lim_{\e\to 0} F_{\e,\delta_\e} (u)=\lim_{T\to +\infty}\Gamma(L^p)-\lim_{\e\to 0} F_{\e,\delta_\e}^T (u),
$$
hence, by Monotone Convergence Theorem, the statement follows once we prove that for every $S\in\rr^{d\times m}$ and $z\in\rr^m$ 
$$
\lim_{T\to +\infty} f_{hom}^T (S)= f_{hom} (S),\quad \lim_{T\to +\infty} \varphi^T (z)= \varphi (z).
$$
The equalities above  are, in turn, again a straightforward consequence of Monotone Convergence Theorem.  
 We may argue analogously in the setting of Theorem \ref{th:mainNL}, taking into account Proposition \ref{captermNL}.

\medskip

\noindent  {\bf  Step 2.} With fixed $T>r_0$, we now prove the validity of the $\Gamma$-$\liminf$ inequality for $F_{\e,\delta_\e}^T$ for both theorems. 

\noindent Given $\e_j\to 0^+$ as $j\to+\infty$, let $u\in W^{1,p}(\Omega;\rr^m)$ and let  $u_j\to u$ in $L^p(\Omega;\rr^m)$ be such that $\sup_j F_{\e_j,\delta_{\e_j}}^T(u_j)<+\infty$.
Up to passing to a subsequence (not relabelled), given $\eta>0$ and $M>0$, we may apply Lemma \ref{truncation} and find $R_M>M$ and a Lipschitz function $\Phi_M:\rr^m\to\rr^m$, with $Lip(\Phi_{M})= 1$, $\Phi_{M} (z) = z$ if $|z| < M$ and $\Phi_{M}(z) = 0$ if $|z| > R_M$ such  that
\begin{equation}\label{stima1}
F_{\e_j,\delta_{\e_j}}^T(u_j)>F_{\e_j,\delta_{\e_j}}^T(\Phi_M(u_j))-\eta.
\end{equation}
Notice that $\phi_M(u_j)\to \Phi_M(u)$ in $L^p(\Omega;\rr^m)$. Given $N\in\NN$, let $\{w_j^M\}_j$ the sequence constructed in Lemma \ref{raccordo} applied with $\rho_j=\delta_{\e_j}/4$ and $\{\Phi_M(u_j)\}_j$ in place of $\{u_j\}_j$.
Set
$$
E_j=\bigcup_{i\in Z_j (\Omega)}B_{\rho_{j,k_i}}(\delta_{\e_j} i),
$$
where $\rho_{j,k_i}$ is defined in Lemma \ref{raccordo}, and define
\begin{equation*} 
v_j^M(x):=
\left\{ 
\begin{array}{ll}
w_j^M (x)&\hbox{ if } x\in\Omega\setminus E_j\\
(\Phi_M(u_j))^{i,k_i}&\hbox{ if } x\in B_{\rho_{j,k_i}}(\delta_{\e_j} i).\\
\end{array}
\right.
\end{equation*}
Notice that, by (\textbf{H}), (G0) and Lemma \ref{raccordo}, $\sup_j G^{r_0,p}_{\e_j}(v_j^M,\Omega)<+\infty$, Hence, by  Theorem \ref{kolcom}, $\{v_j^M\}_j$ is relatively compact in $L^p(\Omega;\rr^m)$. Arguing as in \cite{AB2}, we  now show that $v_j^M\to \Phi_M(u)$  in $L^p(\Omega;\rr^m)$. Specifically, for $h\in\{0,\dots, N-1\}$ set
$$
r_h:=\frac{3}{4}2^{-h-2},\quad \chi_j^h(x):=\chi^h\left(\frac{x}{\delta_{\e_j}}\right),
$$
where  $\chi^h$ coincides with $\chi_{Q_1\setminus B_{r_h}}$ on $Q_1$ and is extended $Q_1$-periodically in $\rr^d$,
$$
Z_j^h:=\{i\in Z_j(\Omega): k_i=h\},\quad D_j^h:=\bigcup_{i\in Z_j^h}\delta_{\e_j}\left( i+Q_1\right),\quad \psi_j^h(x):=\chi_{D_j^h}(x).
$$
Recall that
$$
\chi_j^h\stackrel{*}\rightharpoonup m_h:=| Q_1\setminus B_{r_h}|>0\quad \hbox{weakly* in } L^\infty(\rr^d)
$$
and note that,
$$
\sum_{h=0}^{N-1}\psi_j^h\to 1\ \mbox{strongly in } L^1(\Omega).
$$
Moreover, since $\chi_j^h\geq\chi_j^0$ for every $h\in\{0,\dots, N-1\}$, we have that
\begin{equation}\label{stimaperiodica}
\chi_{\Omega\setminus E_j}=\sum_{h=0}^{N-1}\psi_j^h\chi_j^h\geq\chi_j^0 \sum_{h=0}^{N-1}\psi_j^h\stackrel{*}\rightharpoonup m_0>0.
\end{equation}
Let us consider a subsequence (not relabelled) such that $ \chi_{\Omega\setminus E_j}\stackrel{*}\rightharpoonup g$ in $ L^\infty(\Omega)$ and $ v_j^M\to v$ strongly in  $L^p(\Omega;\rr^m)$.  Hence
$$
\chi_{\Omega\setminus E_j} v_j^M\rightharpoonup g\, v,\quad \chi_{\Omega\setminus E_j} w_j^M\rightharpoonup g\, \Phi_M(u)\quad \hbox{weakly in } L^p(\Omega;\rr^m). 
$$
Taking into account that $\chi_{\Omega\setminus E_j} v_j^M= \chi_{\Omega\setminus E_j} w_j^M$, we conclude that $v=\Phi_M(u)$ thanks to the lower bound  on $g$   ensured by  \eqref{stimaperiodica}. 

\noindent By Lemma \ref{raccordo} and \eqref{stima1}, we  have
\begin{equation}\label{stimaspezzata}
\begin{split}
F_{\e_j,\delta_{\e_j}}^T(u_j) & \ge F_{\e_j,\delta_{\e_j}}^T(\Phi_M(u_j))-\eta\ge \mathcal{F}_{\e_j}^T(w_j^M)-\frac{C}{N}-\eta\\ 
&  \ge \mathcal{F}_{\e_j}^T(w_j^M, \Omega\setminus{E_j})+\mathcal{F}_{\e_j}^T(w_j^M,E_j)-\eta-\frac{C}{N} \\
& \ge \mathcal{F}_{\e_j}^T(v_j^M)+\mathcal{F}_{\e_j}^T(w_j^M,E_j)-\eta-\frac{C}{N}, 
\end{split}
\end{equation}
since by definition  $w_j^M=v_j^M$ in $\Omega\setminus{E_j}$, and  $v_j^M$ is constant on each $\partial^{\e_jT} B_{\rho_{j,k_i}}(\delta_{\e_j} i)$.  

\noindent By Theorem \ref{convexhom}, it holds
\begin{equation}\label{liminfbulk}
\liminf_{j\to +\infty}\mathcal{F}_{\e_j}^T(v_j^M)\geq \int_\Omega f_{hom}^T(\nabla \Phi_M(u))\, dx.
\end{equation}
We now turn to the estimate of the contribution on $E_j$. At this point we need to distinguish whether  \eqref{zero} or \eqref{alfa} holds.
\smallskip

\noindent {\bf Case }$\boldsymbol{\e=o(r_{\delta_{\e}})}$.   
With fixed $i\in Z_j(\Omega)$, let 
$$
v_{j,i}^M(y):=w_j^M(\delta_{\e_j} i+r_{\delta_{\e_j}}y)
$$
be defined on the ball $B_{R_j^i}$, with $R_j^i:= \rho_{j,k_i}r_{\delta_{\e_j}}^{-1}$  and extended to $(\Phi_M(u_j))^{i,k_i}$ outside this ball.
Setting $s_j:=\e_j r_{\delta_{\e_j}}^{-1}$, we get
\begin{equation}\label{stimaprima}
\mathcal{F}_{\e_j}^T(w_j^M, B_{\rho_{j,k_i}}(\delta_{\e_j} i))=r_{\delta_{\e_j}}^{d-p}\mathcal{F}_{s_j}^T(v_{j,i}^M, B_{R_j^i}) \ge r_{\delta_{\e_j}}^{d-p}\varphi_{s_j, T,R_j^i}(\Phi_M(u_j))^{i,k_i}). 
\end{equation}
 We take 
$$
\Psi_j^T(x)= \sum_{i\in Z_j(\Omega)} \varphi_{s_j,T, R_j^i}((\Phi_M(u_j))^{i,k_i})\chi_{Q_{\delta_{\e_j}}(i)}(x).
$$
By \eqref{stimaprima}, we get
\begin{equation}\label{stimapproxcap}
\mathcal{F}_{\e_j}^T(w_j^M, E_j)\geq \frac{r_{\delta_{\e_j}}^{d-p}}{\delta_{\e_j}^d}\int_\Omega \Psi_j^T(x)\, dx=
\left(\frac{r_{\delta_{\e_j}}}{\delta_{\e_j}^{\frac{d}{d-p}}}\right)^{d-p}\int_\Omega \Psi_j^T(x)\, dx.
\end{equation}
By Proposition \ref{th:riemannsum} $(i)$, applied to $(\Phi_M(u_j))$, with $s_j,\ R_j^i$ in place of $\e_j, R_j$, respectively, we have 
\begin{equation}\label{convdens}
\Psi_j^T\to \varphi^T(\Phi_M(u))\ \mbox{in}\ L^1(\Omega).
\end{equation}
Hence, by \eqref{stimapproxcap},  we deduce that
\begin{equation}\label{liminfcap}
\liminf_{j\to +\infty}\mathcal{F}_{\e_j}^T(w_j^M, E_j)\geq \beta^{d-p}\int_\Omega \varphi^T(\Phi_M(u)) \, dx.
\end{equation}
\smallskip

\noindent{\bf Case } $\boldsymbol{\e=O(r_{\delta_{\e}}) }$. 
With fixed $i\in Z_j(\Omega)$, we now set
$$
v_{j,i}^M(y):=w_j^M\Big(\delta_{\e_j} i+\frac{\e_j}{\alpha}y\Big)
$$
on the ball $B_{R_j^i}$, with $R_j^i:=\alpha \rho_{j,k_i} \e_j^{-1}$, and extend it, as in the previous case,  to $(\Phi_M(u_j))^{i,k_i}$ outside this ball.
We get
\begin{equation}\label{stimaprimaalfa}
\mathcal{F}_{\e_j}^T(w_j^M, B_{\rho_{j,k_i}}(\delta_{\e_j} i))=\left(\frac{\e_j}{\alpha}\right)^{d-p}\mathcal{F}_{\alpha}^T(v_{j,i}^M, B_{R_j^i}).
\end{equation}
Define $t_j:=\alpha \frac{ r_{\delta_{\e_j}}}{\e_j}$ and note that, by \eqref{alfa}, $t_j\to 1$. We take 
\begin{equation*} 
\tilde v_{j,i}^M(y):=
\left\{ 
\begin{array}{ll}
0 &\hbox{ if } y\in B_1\\
v_{j,i}^M(y)&\hbox{if } y\in B_{R^i_j}\setminus B_1.\\
\end{array}
\right.
\end{equation*}
Notice that, if $t_j>1$, then $\tilde v_{j,i}^M$ coincides with $ v_{j,i}^M$. A straightforward computation shows that
\begin{equation}\label{errorino}
\mathcal{F}_{\alpha}^T(v_{j,i}^M, B_{R_j^i})\geq \mathcal{F}_{\alpha}^T(\tilde v_{j,i}^M, B_{R_j^i}) -C(M) |1-t_j|.
\end{equation}
Since $\tilde v_{j,i}^M\in L^p_{\alpha, T, (\Phi_M(u_j))^{i,k_i}}(B_{R_j^i};\rr^m)$, by \eqref{stimaprimaalfa} and \eqref{errorino}, we get
\begin{equation}\label{stimaseconda}
\begin{split}
\mathcal{F}_{\e_j}^T(w_j^M, B_{\rho_{j,k_i}}(\delta_{\e_j} i))&\geq \left(\frac{\e_j}{\alpha}\right)^{d-p}(\mathcal{F}_{\alpha}^T(\tilde v_{j,i}^M, B_{R_j^i})-C(M) |1-t_j| )\\
&\geq \left(\frac{\e_j}{\alpha}\right)^{d-p}(\varphi_{\alpha,T, R_j^i}((\Phi_M(u_j))^{i,k_i}) - C(M) |1-t_j| ).
\end{split}
\end{equation}
By taking
$$
\Psi_{j,\alpha}^T(x)= \sum_{i\in Z_j(\Omega)} \varphi_{\alpha, T, R_j^i}((\Phi_M(u_j))^{i,k_i})\chi_{Q_{\delta_{\e_j}}(i)}(x),
$$
using \eqref{stimaseconda}, we obtain
\begin{equation}\label{stimapproxcapalfa}
\mathcal{F}_{\e_j}^T(w_j^M, E_j)\geq t_j^{p-d}(\beta^{d-p}+o(1))\int_\Omega \Psi_{j,\alpha}^T(x)\, dx+o(1).
\end{equation}
Thanks to Proposition \ref{th:riemannsum} $(ii)$, applied to $(\Phi_M(u_j))$, with $\ R_j^i$ in place of $R_j$, we have 
\begin{equation}\label{convdensalfa}
\Psi_{j,\alpha}^T\to \varphi_{NL,\alpha}^T(\Phi_M(u))\ \mbox{in}\ L^1(\Omega).
\end{equation}
Hence, by \eqref{stimapproxcapalfa},  we deduce that
\begin{equation}\label{liminfcapalfa}
\liminf_{j\to +\infty}\mathcal{F}_{\e_j}^T(w_j^M, E_j)\geq \beta^{d-p}\int_\Omega \varphi_{NL,\alpha}^T(\Phi_M(u)) \, dx,
\end{equation}
which is the analogue  of \eqref{liminfcap} in the previous case.
\smallskip

\noindent By \eqref{stimaspezzata}, \eqref{liminfbulk}, together with \eqref{liminfcap} and \eqref{liminfcapalfa}, and by the arbitrariness of $\eta>0$ and $N\in\NN$, we infer that
$$
\liminf_{j\to +\infty} F_{\e_j,\delta_{\e_j}}^T(u_j)\geq \int_\Omega f_{hom}^T(\nabla \Phi_M(u))\, dx+\beta^{d-p}\int_\Omega \varphi^T(\Phi_M(u)) \, dx,
$$
under the assumption \eqref{zero}, and
$$
\liminf_{j\to +\infty} F_{\e_j,\delta_{\e_j}}^T(u_j)\geq \int_\Omega f_{hom}^T(\nabla \Phi_M(u))\, dx+\beta^{d-p}\int_\Omega \varphi^T_{NL,\alpha}(\Phi_M(u)) \, dx,
$$
if \eqref{zero} is replaced by \eqref{alfa}. Letting $M\to +\infty$, we conclude that, under the assumptions of Theorem \ref{th:main},
$$
\liminf_{j\to +\infty} F_{\e_j,\delta_{\e_j}}^T(u_j)\geq \int_\Omega f_{hom}^T(\nabla u)\, dx+\beta^{d-p}\int_\Omega \varphi^T(u) \, dx,
$$
and, under the assumption of  Theorem \ref{th:mainNL}, 
$$
\liminf_{j\to +\infty} F_{\e_j,\delta_{\e_j}}^T(u_j)\geq \int_\Omega f_{hom}^T(\nabla u)\, dx+\beta^{d-p}\int_\Omega \varphi^T_{NL,\alpha}(u) \, dx.
$$

\medskip

\noindent {\bf  Step 3.} With fixed $T\geq r_0$, we now prove the validity of the $\Gamma$-$\limsup$ inequality for $F_{\e,\delta_\e}^T$.
By a density argument it suffices to prove the inequality for $u\in C_c^\infty (\rr^d;\rr^m)$. For such a $u$, fixed an open set $\Omega'\in \mathcal{A}^{\rm reg}(\rr^d)$ such that $\Omega'\supset\supset\Omega$, and given $\e_j\to 0$ as $j\to +\infty$, by Theorem \ref{convexhom} there exists a sequence $(\tilde u_j)$, converging in $L^p(\Omega'; \rr^m)$ to $u$,   such that
\begin{equation}\label{recovery}
\lim_{j\to +\infty}\mathcal{F}_{\e_j}^T(\tilde u_j, \Omega')=\int_{\Omega'} f_{hom}^T(\nabla u)\, dx.
\end{equation} 
Taking into account Lemma \ref{truncation}, up to replacing $\tilde u_j$ with a suitable truncation, we may also assume that $\sup_j \|\tilde u_j\|_{L^\infty (\Omega'; \rr^m)} <+\infty$.  Thus, given $N\in\NN$, we consider the sequence  $(w_j)$  constructed in Lemma \ref{raccordo} applied with $\Omega'$ in place of $\Omega$, $\rho_j=\delta_{\e_j}/4$ and $\tilde u_j$ in place of $u_j$, so that
\begin{equation}\label{stimasupbulk}
\mathcal{F}_{\e_j}^T(w_j, \Omega')\leq \mathcal{F}_{\e_j}^T(\tilde u_j, \Omega') +\frac CN.
\end{equation}
We now pass to the estimate of the energetic contribution on the set
$$
E_j=\bigcup_{i\in Z_j(\Omega')}B_{\rho_{j,k_i}}(\delta_{\e_j} i).
$$
As previously done, we  distinguish whether  \eqref{zero} or \eqref{alfa} holds.
\smallskip

\noindent {\bf Case }$\boldsymbol{\e=o(r_{\delta_{\e}})}$.   
Set
$$
R_j^i:= \frac{\rho_{j,k_i}}{r_{\delta_{\e_j}}},\quad s_j=\frac{\e_j}{r_{\delta_{\e_j}}},
$$
where $\rho_{j,k_i}$ is defined in Lemma \ref{raccordo}. For $i\in Z_j(\Omega')$, let $\tilde v_{j,i}\in L^p_{s_j,T, u^{i,k_i}_j}(B_{R_j^i};\rr^m)$ be such that
$$
\mathcal {F}_{s_j}^T(\tilde v_{j,i}, B_{R_j^i})=\varphi_{s_j,T, R_j^i}(u^{i,k_i}_j) +o(\e_j).
$$
Then set 
$$
v_{j,i}(x):=\tilde v_{j,i}\left(\frac {x-\delta_{\e_j} i} {r_{\delta_{\e_j}} } \right),\  \ x\in B_{\rho_{j,k_i}}(\delta_{\e_j} i),
$$
and
\begin{equation*} 
u_j(x):=
\left\{ 
\begin{array}{ll}
w_j (x) &\hbox{ if } x\in\Omega' \setminus E_j\\
v_{j,i} (x)&\hbox{ if } x\in B_{\rho_{j,k_i}}(\delta_{\e_j} i).\\
\end{array}
\right.
\end{equation*}
Note that
\begin{equation}\label{energiapalle}
\mathcal{F}_{\e_j}^T(v_{j,i}, B_{\rho_{j,k_i}}(\delta_{\e_j} i))=r^{d-p}_{\delta_{\e_j}}\mathcal {F}_{s_j}^T(\tilde v_{j,i}, B_{R_j^i})=r^{d-p}_{\delta_{\e_j}}(\varphi_{s_j,T, R_j^i}(u^{i,k_i}_j)+o(\e_j)).
\end{equation}
Moreover $u_j\in L^p_{\delta_{\e_j}}(\Omega' ;\rr^m)$ and, arguing as in \textbf{Step 2}, we also deduce that $u_j\to u$ in $L^p(\Omega;\rr^m)$. We finally pass to the estimate of the energy. By \eqref{energiapalle}, we get
\begin{equation}\label{stimasupenergia}
\begin{split}
\mathcal{F}_{\e_j}^T(u_j)&\leq \mathcal{F}_{\e_j}^T(w_j, \Omega')+\sum_{i\in Z_j(\Omega')}\mathcal{F}_{\e_j}^T(u_j, B_{\rho_{j,k_i}}(\delta_{\e_j} i))\\
&=\mathcal{F}_{\e_j}^T(w_j, \Omega' )+\sum_{i\in Z_j(\Omega')}r^{d-p}_{\delta_{\e_j}}(\varphi_{s_j,T, R_j^i}(u^{i,k_i}_j) +o(\e_j))
\end{split}
\end{equation}
Applying Proposition \ref{th:riemannsum} $(i)$ as in {\it Step 2}, we deduce that
\begin{equation}\label{riemannsum}
\sum_{i\in Z_j(\Omega')}r^{d-p}_{\delta_{\e_j}}\varphi_{s_j,T, R_j^i}(u^{i,k_i}_j)\to \beta^{d-p}\int_{\Omega'} \varphi^T(u)\, dx.
\end{equation}
\smallskip 

\noindent {\bf Case }$\boldsymbol{\e=O(r_{\delta_{\e}})}$.   
We now set
$$
R_j^i:= \alpha\, \frac{\rho_{j,k_i}}{\e_j},
$$
where $\rho_{j,k_i}$ is defined in Lemma \ref{raccordo}. By applying Proposition \ref{potcaplimitato} with $M_0=\sup_j \|\tilde u_j\|_{L^\infty (\Omega'; \rr^m)}$, given $\eta >0$ there exist $M>\sup_j \|\tilde u_j\|_{L^\infty (\Omega'; \rr^m)}$ such that for every $i\in Z_j(\Omega')$ there exists $\tilde v_{j,i}\in L^p_{\alpha,T, u^{i,k_i}_j}(B_{R_j^i};\rr^m)$   such that $\|\tilde v_{j,i}\|_{L^\infty(\Omega'; \rr^m)}\leq M$ and 
$$
\mathcal {F}_{\alpha}^T(\tilde v_{j,i}, B_{R_j^i})\leq \varphi_{\alpha, T, R_j^i}(\tilde u^{i,k_i}_j)+  \eta.
$$
Set $t_j=\alpha \frac{r_{\delta_{\e_j}}}{\e_j}$ and note that, by \eqref{alfa}, $t_j\to 1$. Then define
\begin{equation*} 
\hat v_{j,i}(y):=
\left\{ 
\begin{array}{ll}
0 &\hbox{ if } y\in B_{t_j}\\
\tilde v_{j,i}(y)&\hbox{if } y\in B_{R^i_j}\setminus B_{t_j}.\\
\end{array}
\right.
\end{equation*}
Notice that $\hat v_{j,i}$ coincides with $\tilde v_{j,i}$ if $t_j\leq 1$. As in {\it Step 2}, a straightforward computation shows that
$$
\mathcal {F}_{\alpha}^T(\hat v_{j,i}, B_{R_j^i})\leq \mathcal {F}_{\alpha}^T(\tilde v_{j,i}, B_{R_j^i})+C(M)|1-t_j|
$$
We take 
$$
v_{j,i}(x):=\hat v_{j,i}\left(\alpha\, \frac {x-\delta_{\e_j} i} {\e_j} \right) \ x\in B_{\rho_{j,k_i}}(\delta_{\e_j} i),
$$
and 
\begin{equation*} 
u_j(x):=
\left\{ 
\begin{array}{ll}
w_j (x) &\hbox{ if } x\in\Omega'\setminus E_j\\
v_{j,i} (x)&\hbox{ if } x\in B_{\rho_{j,k_i}}(\delta_{\e_j} i).\\
\end{array}
\right.
\end{equation*}
Note that
\begin{equation}\label{energiapallealfa}
\mathcal{F}_{\e_j}^T(v_{j,i}, B_{\rho_{j,k_i}}(\delta_{\e_j} i))=t_j^{p-d}r^{d-p}_{\delta_{\e_j}}\mathcal {F}_{\alpha}^T(\hat v_{j,i}, B_{R_j^i})\leq t_j^{p-d}r^{d-p}_{\delta_{\e_j}}(\varphi_{\alpha,T,R_j^i}(\tilde u^{i,k_i}_j)+C(M)|1-t_j|+\eta).
\end{equation}
We notice that $u_j\in L^p_{\delta_{\e_j}}(\Omega;\rr^m)$ and, as in the previous case,  we deduce that $u_j\to u$ in $L^p(\Omega;\rr^m)$.  By \eqref{energiapallealfa}, we get the counterpart of \eqref{stimasupenergia}
\begin{equation}\label{stimasupenergiaalfa}
\begin{split}
\mathcal{F}_{\e_j}^T(u_j)&\leq \mathcal{F}_{\e_j}^T(w_j, \Omega' )+\sum_{i\in Z_j(\Omega')}\mathcal{F}_{\e_j}^T(u_j, B_{\rho_{j,k_i}}(\delta_{\e_j} i))\\
&\leq \mathcal{F}_{\e_j}^T(w_j,\Omega')+t_j^{p-d}\sum_{i\in Z_j(\Omega')}r^{d-p}_{\delta_{\e_j}}(\varphi_{\alpha,T,R_j^i}(\tilde u^{i,k_i}_j)+C|1-t_j|+\eta).
\end{split}
\end{equation}
Applying Proposition \ref{th:riemannsum} $(ii)$ as in \textbf{Step 2}, we deduce that
\begin{equation}\label{riemannsumalfa}
\sum_{i\in Z_j(\Omega ')}r^{d-p}_{\delta_{\e_j}}\varphi_{\alpha, T,R_j^i}(\tilde u^{i,k_i}_j)\to \beta^{d-p}\int_{\Omega'} \varphi_{NL,\alpha}^T(u)\, dx.
\end{equation}
which corresponds to \eqref{riemannsum}.
\smallskip

\noindent Hence, by \eqref{recovery}, \eqref{stimasupbulk}, together with \eqref{stimasupenergia}, \eqref{riemannsum}, \eqref{stimasupenergiaalfa}, and \eqref{riemannsumalfa},  we get that
$$
\limsup_{j\to +\infty}\mathcal{F}_{\e_j}^T(u_j)\leq \int_{\Omega'}f_{hom}^T(\nabla u)\, dx+\beta^{d-p}\int_{\Omega'} \varphi^T(u)\, dx+\frac CN,
$$
under the assumptions of Theorem \ref{th:main}, and
$$
\limsup_{j\to +\infty}\mathcal{F}_{\e_j}^T(u_j)\leq \int_{\Omega'}f_{hom}^T(\nabla u)\, dx+\beta^{d-p}\int_{\Omega'} \varphi^T_{NL,\alpha}(u)\, dx+\frac CN,
$$
under the assumptions of Theorem \ref{th:mainNL}. The conclusion follows by the arbitrariness of $N\in \NN$ and letting $\Omega' \to \Omega$.

\section{The scaling regime $r_{\delta_\e}=o(\e)$}\label{sopracritico}

Theorems \ref{th:main}, \ref{th:mainNL} and Remark \ref{altrescale}  provide a complete  description of the asymptotic behaviour of $F_{\e,\delta_\e}$
when $\displaystyle \lim_{\e\to 0}\e r_{\delta_\e}^{-1}=\alpha\in [0,+\infty)$. In this section we consider the case when $\e\to 0$ slower than $r_{\delta_\e}$, showing that, for most of the choice of the scaling of $r_{\delta}$ with respect to $\delta$, $F_{\e,\delta_\e}$ is not affected by the constraint $u\in L^p_{\delta_\e}(\Omega;\rr^m)$ and then $\displaystyle\Gamma(L^p)-\lim_{\e\to 0} F_{\e,\delta_\e} (u)=\Gamma(L^p)-\lim_{\e\to 0} \mathcal{F}_\e (u)$. 

Assume from now on
\begin{equation}\label{buchetta}
\lim_{\delta\to 0} \frac{r_\delta}{\delta}=0,
\end{equation}

\begin{equation}\label{altroregime}
\lim_{\e\to 0}\frac{\e}{ r_{\delta_\e}}=+\infty.
\end{equation}

First of all let us note that, since $F_{\e,\delta_\e}(u)\geq \mathcal{F}_\e (u)$, we have that
\begin{equation}\label{gammastimainf}
\Gamma(L^p)-\liminf_{\e\to 0} F_{\e,\delta_\e}(u)\geq\Gamma(L^p)-\liminf_{\e\to 0} \mathcal{F}_\e (u).
\end{equation}
Given $\e_j\to 0$ as $j\to +\infty$, let $u\in C^\infty_c(\rr^d;\rr^m)$ and let $\tilde u_j\to u$ in $L^p(\Omega;\rr^m)$ such that 
$$
\lim_{j\to +\infty}\mathcal{F}_{\e_j} (\tilde u_j)=\int_{\Omega} f_{\hom}(\nabla u)\, dx.
$$
Arguing as in Step 3 of the proof of Theorems \ref{th:main} and \ref{th:mainNL}, we may assume $\tilde u_j$ bounded in $L^\infty(\Omega;\rr^m)$.

Set then
\begin{equation*} 
u_j(x):=
\left\{ 
\begin{array}{ll}
\tilde u_j (x) &\hbox{ if } x\in\Omega\setminus P_{\delta_{\e_j}}\\
0 &\hbox{ if } x\in \Omega\cap P_{\delta_{\e_j}}.\\
\end{array}
\right.
\end{equation*}

Clearly $u_j\in L^p_{\delta_{\e_j}}(\Omega;\rr^m)$ and, by \eqref{buchetta}, $u_j\to u$ in $L^p(\Omega;\rr^m)$ as $j\to +\infty$. By assumption (ii), we have
\begin{equation}\label{serve}
\mathcal{F}_{\e_j}(u_j)\leq \mathcal{F}_{\e_j} (\tilde u_j)+ \frac{\|u_j\|_\infty^p}{\e_j^p}\sum_{i\in\ZZ^d}\int_{\rr^d} M(\xi)|S_{i,j}^\xi|\,d\xi,
\end{equation}
where
$$
S_{i,j}^\xi:=\Omega\cap\Bigl((B_{r_{\delta_{\e_j}}}(\delta_{\e_j} i)\cap ((B_{r_{\delta_{\e_j}}}(\delta_{\e_j} i))^c-\e_j\xi))\cup ((B_{r_{\delta_{\e_j}}}(\delta_{\e_j} i))^c\cap (B_{r_{\delta_{\e_j}}}(\delta_{\e_j} i)-\e_j\xi))\Bigr).
$$
Note that,
since $S_{i,j}^\xi\subseteq B_{r_{\delta_{\e_j}}}(\delta_{\e_j} i)\cup (B_{r_{\delta_{\e_j}}}(\delta_{\e_j} i)-\e_j\xi)$, then
$$
|S_{i,j}^\xi|\leq C r_{\delta_{\e_j}}^{d}.
$$
Thus, by \eqref{serve} and (G1), we get
\begin{equation}\label{supest}
\mathcal{F}_{\e_j}(u_j)\leq \mathcal{F}_{\e_j} (\tilde u_j)+ C \frac{r_{\delta_{\e_j}}^{d}}{\e_j^p}\frac{1}{\delta_{\e_j}^d}.
\end{equation}
By \eqref{gammastimainf},\eqref{supest} and a density argument, we infer that $\displaystyle\Gamma(L^p)-\lim_{\e\to 0} F_{\e,\delta_\e} (u)=\Gamma(L^p)-\lim_{\e\to 0} \mathcal{F}_\e (u)$ under the additional condition
\begin{equation}\label{addcond}
\lim_{\e\to 0 }\frac{r_{\delta_{\e}}^{d}}{\e^p}\frac{1}{\delta_{\e}^d}=0,
\end{equation}
which can be written as
$$
\lim_{\e\to 0 }\frac{\e}{\left(\frac{r_{\delta_\e}}{\delta_\e}\right)^{d/p}}=+\infty.
$$
Note that
$$
\frac{r_{\delta_{\e}}^{d}}{\e^p}\frac{1}{\delta_{\e}^d}=\left(\frac{r_{\delta_{\e}}}{\e}\right)^p\frac{r_{\delta_{\e}}^{d-p}}{\delta_{\e}^d},
$$
which, thanks to \eqref{altroregime}, yields that \eqref{addcond} is satisfied if $r_{\delta}\leq C \delta ^{\frac{d}{d-p}}$. We may then conclude that the following $\Gamma$-convergence result holds.

\begin{theorem}\label{th:sloweps}
Let $F_{\e,\delta}$ be defined by \eqref{functionals2}, with $f$ satisfying assumptions {\rm ({\bf H})}, {\rm ({\bf G})} and {\rm ({\bf L})} and $1<p<d$. Assume moreover that  \eqref{altroregime} and one of the following two assumptions hold
\begin{itemize}
\item[a)] $\displaystyle\limsup_{\delta\to 0}\frac{r_\delta}{\delta^{\frac{d}{d-p}}}<+\infty$

\item[b)] $\displaystyle\lim_{\e\to 0 }\frac{\e}{\left(\frac{r_{\delta_\e}}{\delta_\e}\right)^{d/p}}=+\infty.$

\end{itemize}
Then
\[
\Gamma(L^p)\hbox{-}\lim_{\e\to 0} F_{\e,\delta_\e} (u)=
\begin{cases}\displaystyle \int_\Omega f_{hom}(\nabla u)\, dx & \text{if}\ u\in W^{1,p}(\Omega;\mathbb{R}^m),\cr\
+\infty & \text{otherwise},
\end{cases}
\]
where $f_{hom} (S)$ is defined by \eqref{homform}.
\end{theorem} 
\medskip

We have summarized our $\Gamma$-convergence results in the following table (see Theorem \ref{th:main}, Theorem \ref{th:mainNL}, Remark \ref{altrescale}, and Theorem \ref{th:sloweps}), schematising how the interplay between the various parameters affects the $\Gamma$-limit  of the non-local functionals $F_{\e,\delta}$ defined in \eqref{functionals2}. The domain  of the $\Gamma$-limit is $W^{1,p}(\Omega;\rr^m)$ if not specified.

\bigskip

\begin{table} [h!]
		\centering
		\small
		{\renewcommand\arraystretch{1.2} 
		\begin{tabular}{| l | c | c | }
			\hline
			\diagbox[innerwidth=2.5cm,innerleftsep=1cm,innerrightsep=1cm]{$\e$}{$\delta$} & $\displaystyle \lim_{\delta\to 0} \frac{r_\delta}{\delta^{\frac {d}{d-p}}}=\beta\ge 0$ & $\displaystyle \lim_{\delta\to 0} \frac{r_\delta}{\delta^{\frac {d}{d-p}}}=+\infty$\\ [0.5ex]
			 \hline
			 \rule[-5mm]{0mm}{1,1cm}
			  $\displaystyle\lim_{\e\to 0}\frac{\e}{r_{\delta_\e}}=0$ & $\displaystyle \int_\Omega f_{hom}(\nabla u)\, dx +\beta^{d-p}\int_\Omega \varphi(u)\, dx$ & 	$0$ iff $u\equiv 0$ \\
			  \hline		 
			\rule[-5mm]{0mm}{1,1cm}
			 $\displaystyle\lim_{\e\to 0}\frac{\e}{r_{\delta_\e}}=\alpha>0$ & $\displaystyle \int_\Omega f_{hom}(\nabla u)\, dx +\beta^{d-p}\int_\Omega \varphi_{NL,\alpha}(u)\, dx$ & 	$0$ iff $u\equiv 0$ \\
			 \hline
			  \rule[-8mm]{0mm}{1,8cm}
			   $\displaystyle\lim_{\e\to 0}\frac{\e}{r_{\delta_\e}}=+\infty$ & $\displaystyle \int_\Omega f_{hom}(\nabla u)\, dx$ & 
			   {\minitab[c]{if $\displaystyle \lim_{\e\to 0}\frac{1}{\e}\left(\frac{r_{\delta_\e}}{\delta_\e}\right)^{\frac{d}{p}}=0$  \\ $\displaystyle \int_\Omega f_{hom}(\nabla u)\, dx$} }\\
			\hline
			\end{tabular}}
			\caption{}
			\end{table}

\begin{ack}

The authors are members of Gruppo Nazionale per l'Analisi Ma\-te\-ma\-ti\-ca, la Probabilit\`a e le loro Applicazioni (GNAMPA) of INdAM.
The research of R. Alicandro has been supported by PRIN project 2022J4FYNJ ``Variational methods for stationary and evolution problems with singularities and interfaces", 
the research of M.S. Gelli and C. Leone by PRIN Project 2022E9CF89 ``Geometric Evolution Problems and Shape Optimizations'', the research of M.S. Gelli by PRIN PNRR Project  P2022WJW9H
``Magnetic skyrmions, skyrmionic bubbles and domain walls for spintronic applications". 
PRIN projects are part of PNRR Italia Domani, financed by European Union through NextGenerationEU. 

Part of this work has been done during the Trimester Program ``Mathematics for Complex Materials" (03/01/2023-14/04/2023)
at the Hausdorff Institute for Mathematics (HIM) in Bonn,
funded by the Deutsche Forschungsgemeinschaft (DFG, German Research Foundation)
under Germany's Excellence Strategy - EXC - 2047/1 - 390685813. The authors are grateful to the organizers 
 for their  kind invitation and the nice working atmosphere provided during the whole staying. 
\end{ack}

%

\end{document}